\documentclass[12pt]{article}
\usepackage{graphicx} % Required for inserting images
\usepackage[top=2cm,bottom=2cm,left=2cm,right=2cm]{geometry}

\usepackage{palatino}
\usepackage{latexsym}

\usepackage[mathscr]{eucal}
\usepackage{amsmath}
\usepackage{amsthm}
\usepackage{amsfonts}
\usepackage{amssymb}
\usepackage{amscd}
\usepackage{color}
\usepackage{graphicx}
\usepackage{graphics}
\usepackage{pifont}
\usepackage{subfigure}
\usepackage[makeroom]{cancel}
\usepackage[normalem]{ulem}
\usepackage[dvipsnames]{xcolor}
\usepackage{tikz}
\usepackage{cite}
\usepackage[hidelinks,pdfusetitle]{hyperref}

\theoremstyle{plain}
\newtheorem{theorem}{Theorem }

\newtheorem{lemma}[theorem]{Lemma}
\newtheorem{proposition}[theorem]{Proposition}
\newtheorem{definition}[theorem]{Definition}
\newtheorem{corollary}[theorem]{Corollary}

\newcommand{\Z}{{\mathbb Z}}
\newcommand{\R}{{\mathbb R}}
\newcommand{\N}{{\mathbb N}}
\newcommand{\T}{{\mathbb T}}
\newcommand{\C}{{\mathbb C}}

\title{Small-time approximate controllability of bilinear Schrödinger equations and diffeomorphisms}

\begin{document}
\author{Karine Beauchard\footnote{Univ Rennes, CNRS, IRMAR - UMR 6625, F-35000 Rennes, France (karine.beauchard@ens-rennes.fr)},\quad Eugenio Pozzoli\footnote{Univ Rennes, CNRS, IRMAR - UMR 6625, F-35000 Rennes, France (eugenio.pozzoli@univ-rennes.fr)}}
\maketitle

\abstract{%Recent advances on geometric control theory have shown that arbitrary $L^2$ local phases can be approximately controlled in small time for certain physically relevant Schrödinger PDEs on $\T^d$ and $\R^d$ \cite{duca-nersesyan,duca-pozzoli}. In this paper we build on such results and show that, under the same hypothesis, compositions with large families of diffeomorphisms (more precisely, the ones that can be written as a product of gradient flows) can be approximately controlled in small time, in $L^2$. In particular, in the one-dimensional case $d=1$, all compactly supported and isotopic to the identity diffeomorphisms can be obtained. 
We consider Schrödinger PDEs, posed on a boundaryless Riemannian manifold $M$, with bilinear control. We propose a new method to prove the global $L^2$-approximate controllability. Contrarily to previous ones, it works in arbitrarily small time and does not require a discrete spectrum.

%This approach consists in combining the control of the phase and the control of the group ${\rm Diff}_c^0(M)$ of the diffeomorphisms (isotopic to the identity and with compact support) of the underlying manifold $M$; they refer, respectively, to the possibility, for any initial state $\psi_0\in L^2(M,\C)$, phase $\varphi\in L^2(M,\R)$ and diffeomorphism $P\in {\rm Diff}_c^0(M)$,to reach approximately the states $e^{i \phi}\psi_0 $ and $|J_P|^{1/2}(\psi_0\circ P)$ (here $|J_P|$ denotes the determinant of the Jacobian of $P$). 
%We use them, respectively, to reach the expected angular and radial part of the wavefunction. The possibility of controlling the radial part with the group of diffeomorphisms is a consequence of a result of Moser on the transitivity of the action of ${\rm Diff}_c^0(M)$ on $L^2(M,(0,\infty))$ \kb{plutôt $C^{\infty}_c(M,(0,\infty))$ ou "the set of positive densities"}.

This approach consists in controlling separately the radial part and the angular part of the wavefunction thanks to the control of the group ${\rm Diff}_c^0(M)$ of diffeomorphisms %(isotopic to the identity and with compact support) 
of $M$ and the control of phases, which refer to the possibility, for any initial state $\psi_0\in L^2(M,\C)$, diffeomorphism $P\in {\rm Diff}_c^0(M)$ and phase $\varphi \in L^2(M,\R)$ to reach approximately the states $(\det DP)^{1/2}(\psi_0\circ P)$ and $e^{i \varphi}\psi_0 $.
The control of the radial part uses the transitivity of the group action of ${\rm Diff}_c^0(M)$ on positive densities proved by Moser \cite{moser}.

We develop this approach on two examples of Schr\"odinger equations, posed on $\T^d$ or $\R^d$, for which the small-time control of phases was recently proved.
We prove that it implies the small-time control of flows of vector fields thanks to Lie bracket techniques.
Combining this property with the simplicity of the group ${\rm Diff}_c^0(M)$ proved by Thurston \cite{thurston}, we obtain the control of the group ${\rm Diff}_c^0(M)$.

% and show that the small-time approximate controllability problem for these Schrödinger PDEs posed on some boundaryless manifolds $M$ can be related to the controllability problem in the diffeomorphisms group of $M$.

%The main idea is that we can approximate the small-time behaviour of the solution of some controlled Schr\"odinger equations, with the solution of a transport equation along an arbitrary gradient vector field.

 %(more precisely, the ones that can be written as a product of gradient flows) 

% We then further apply these ideas and provide physically relevant examples of bilinear Schrödinger equations which are globally approximately controllable in arbitrarily small time.}

%%%%%%%%%%%%%%%%%%%%%%%%%%%%%%%%%%%%%%%%%%%%%%%%%%%%
\vspace{0.2cm}
\vspace{0.2cm}


\section{Introduction}

%---------------------------------------------------
\subsection{Models}

We consider Schrödinger equations of the form
\begin{equation}\label{eq:schro}
\begin{cases}
i\partial_t\psi(t,x)=\left(-\Delta+V(x)+\sum_{j=1}^m u_j(t)W_j(x)\right)\psi(t,x), & (t,x) \in (0,T) \times M, \\
\psi(0,\cdot)=\psi_0,
\end{cases}
\end{equation}
where 
$M$ is a smooth connected boundaryless Riemannian manifold,
$\Delta$ is the Laplace-Beltrami operator of $M$,
the functions $V, W_1,\dots,W_m: M \to \R$ are real valued potentials,
and the functions $u_1,\dots,u_m:(0,T) \to \R$ are real valued controls. The time-independent part $-\Delta+V$ is usually referred to as \emph{the drift}. The time-dependent potential $\sum_{j=1}^m u_j(t)W_j(x)$ is possibly unbounded on $L^2(M,\C)$. For a time-dependent function 
$u=(u_1,\dots,u_m)$ and an initial state $\psi_0$
in the unitary sphere 
\begin{equation} \label{def:S}
\mathcal{S}:=\{\psi\in L^2(M,\C)\, ;\, \|\psi\|_{L^2(M)}=1\}.
\end{equation}
then $\psi(t;u,\psi_0)$ denotes - when it is well defined - the solution of \eqref{eq:schro}.

System \eqref{eq:schro} describes the dynamics of a quantum particle on the manifold $M$, with free (kinetic plus potential) energy $-\Delta+V$, in interaction with additional external fields with potentials $W_j$ that can be switched on and off. It is used to model a variety of physical situations, such as atoms in optical cavities \cite{haroche}, and molecular dynamics \cite{rabitz}. In this article, we study in particular two examples of equations of the form \eqref{eq:schro}. 

\paragraph{An equation posed on $M=\T^d=\R^d/2\pi\Z^d$.}

The first example is the equation 
\begin{equation}\label{eq:torus}
\begin{cases}
i\partial_t\psi(t,x)=\Big(-\Delta+V(x
)+\sum\limits_{j=1}^d ( u_{2j-1}(t)\sin +u_{2j}(t)\cos)\langle b_j , x\rangle \Big)\psi(t,x),(t,x)\in(0,T)\times\T^d,\\
\psi(0,\cdot)=\psi_0, &
\end{cases}
\end{equation}
where 
%$(A_1,\dots,A_d)\in C^1(\T^d,\R^d), 
$V\in L^\infty(\T^d,\R)$, and 
\begin{equation} \label{Def:ej}
b_1=(1,0,\dots,0),\quad b_2=(0,1,\dots,0),\quad \dots, \quad b_{d-1}=(0,\dots,1,0),\quad b_d=(1,\dots,1).
\end{equation}
In the one-dimensional case $d=1$, it describes the orientation in the plane of a rigid molecule, controlled by the dipolar interactions with two electric fields of constant orthogonal directions and time-variable amplitudes. It is widely used in physics and chemistry as a model for rotational molecular dynamics (see, e.g., the recent review \cite{koch} and the references therein). 
%Recently, small-time approximate controllability properties of \eqref{eq:torus} were obtained in \cite{duca-nersesyan,coron-xiang-zhang}. 

%- - - - - - - - - - - - - - - - - - - -
\paragraph{An equation posed on $M=\R^d$.}

The second example is the equation 
\begin{equation}\label{eq:line}
\begin{cases}
i\partial_t\psi(t,x)=\left( -\Delta+V(x)+\sum\limits_{j=1}^d u_{j}(t)x_j+u_{d+1}(t)e^{-|x|^2/2}  \right)\psi(t,x),& (t,x) \in (0,T) \times \R^d, \\
\psi(0,\cdot)=\psi_0, &  
\end{cases}
\end{equation}
where 
\begin{equation} \label{Hyp:V_transp}
  V\in L^2_{\rm loc}(\R^d,\R) \quad \text{ and } \quad
  \exists a,b\geq 0,\, \forall x \in \R^d,\, V(x)\geq -a |x|^2-b
\end{equation}
It models the dipolar interaction of a quantum particle with controls coupling to its positions $x_j$, and an additional control concentrated around the origin as a Gaussian function. Notice that such system, when e.g. $V=0$, has purely continuous spectrum. 
%Recently, small-time approximate controllability properties of (\ref{eq:line}) were obtained in \cite{duca-pozzoli}. 

\medskip

In this article, we use piecewise constant controls $u$, then the solutions of (\ref{eq:torus}) and (\ref{eq:line}) are well defined and $\psi(.;u,\psi_0) \in C^0([0,T],\mathcal{S})$ (see Appendix \ref{App:WP}).
This is where the assumption (\ref{Hyp:V_transp}) comes in.

%---------------------------------------------------
\subsection{Controllability problems, litterature review}

The mathematical control theory of bilinear Schrödinger PDEs as (\ref{eq:schro}) has undergone a vast development in the last two decades. Such theoretical problems find their origins in applications of quantum control to physics and chemistry (e.g. absorption spectroscopy) \cite{glaser}, or computer science (e.g. quantum computation) \cite{preskill}.

\medskip

The wavefunction is defined up to global phases (the state $e^{i\theta}\psi_1$ for some constant $\theta\in \R$, is physically the same as $\psi_1$) thus we can adopt the following definition.

\begin{definition}[Exact/Approximate controllability]
We say that \eqref{eq:schro} is 
exactly controllable in $L^2$ 
(resp. $L^2$-approximately controllable)
if, for every 
$\psi_0, \psi_1 \in \mathcal{S}$ 
(resp. $\varepsilon >0$), 
there exist 
a time $T>0$,
a global phase $\theta \in [0,2\pi)$
and a 
control $u:[0, T ]\to \R^m$ such that 
$\psi( T;u,\psi_0)= \psi_1 e^{i\theta}$
(resp. $\| \psi( T;u,\psi_0) - \psi_1 e^{i\theta} \|_{\mathcal{H}} <\epsilon$).
\end{definition}

By the seminal work \cite{BMS}, 
if the drift $i(\Delta-V)$ generates a group of bounded operators on an Hilbert space $\mathcal{H}$, on which the control operators $W_j$ are bounded, then
the equation (\ref{eq:schro}) is not exactly controllable in $\mathcal{H}$, because the reachable set has empty interior in $\mathcal{S} \cap \mathcal{H}$ (see also \cite{Chambrion-Caponigro-Boussaid-2020,chambrion-laurent,chambrion-laurent2} for recent developments). This obstruction holds e.g. for the system \eqref{eq:torus}, in any $H^s(\T^d,\C)$, $s\in\N$, with controls in $L^1_{loc}(\R,\R^m)$. 
Therefore different notions of controllability have been investigated. 

\medskip

On the one hand, the exact controllability of (\ref{eq:schro}) was investigated in more regular spaces (i.e. on which the $W_j$ are not bounded). For instance, exact controllability results were proved for 1D-Schrödinger equations with Dirichlet boundary conditions in \cite{beauchard1,beauchard-coron,beauchard-laurent,nersesyan-nersisyan,morancey-polarizability,bournissou}, essentially thanks to linearization techniques.

\medskip

On the other hand, approximate controllability was studied \cite{BCMS,nersesyan,ervedoza}. In particular, if the drift Hamiltonian $-\Delta+V$ has compact resolvent (hence purely point spectrum) then \eqref{eq:schro} is globally $L^2$-approximately controllable in large times with one control ($m=1$), generically w.r.t. the potentials $V,W_1$ \cite{MS-generic}. This holds e.g. for the system \eqref{eq:torus} with $d=1$ and $V=0$ \cite{BCCS}. The proof consists in applying methods from the control theory of ODEs to the Galerkin approximations (i.e., the projections of the system onto finite dimensional eigenspaces of the Hamiltonian) and estimating the error between them and the solution of the PDE.

This strategy does not work when the Hamiltonian does not admit a Hilbert basis of eigenfunctions or presents continuous regions in the spectrum (see \cite{mirrahimi,chambrion,beauchard-coron-rouchon} for different techniques and partial results, e.g. large-time approximate control between bound states). In particular, for system (\ref{eq:line}), even the large-time approximate controllability was not known, due to the purely continuous character of its spectrum, which prevents the applications of the previous techniques.

Moreover, as the previous strategy is based on resonant averaging, it does not work when the control time is small, as the time expected to average out unwilling contributions is in general large.

%---------------------------------------------------
\subsection{The small-time challenge, our main result}

The small-time controllability refers to the possibility of controlling the equation in almost time zero. 

\begin{definition}
%[Small time $L^2$-approximate controllability] 
\label{def:STAC}
We say that \eqref{eq:schro} is \textbf{small-time $L^2$-approximately controllable} ($L^2$-STAC)  if, for every 
$\psi_0, \psi_1\in \mathcal{S}$ and
$\varepsilon >0$, there exist 
a time $T \in[0,\varepsilon]$,
a global phase $\theta \in [0,2\pi)$
and %a locally integrable 
a control $u:[0, T ]\to \R^m$ such that 
%the Cauchy problem \eqref{eq:schro} 
%has a unique solution $\psi \in C^0([0,T],L^2(M))$, and
$\|\psi( T;u,\psi_0)-e^{i\theta} \psi_1\|_{\mathcal{H}}<\varepsilon$.
\end{definition}

Small-time controllability has particularly relevant physical implications, both from a fundamental viewpoint and for technological applications. As a matter of fact, quantum systems, once engineered, suffer of very short lifespan before decaying (e.g., through spontaneous photon emissions) and loosing their non-classical properties (such as superposition). The capability of controlling them in a minimal time is in fact an open challenge also in physics (see, e.g., the pioneering work \cite{khaneja-brockett-glaser} on the minimal control-time for spin systems).

There are however examples of Schrödinger equations of the form (\ref{eq:schro}) which are $L^2$-approximately controllable in large times, but not in small times. This obstruction happens e.g. when $M=\R^d$ in the presence of (sub)quadratic potentials, because Gaussian states are preserved, at least for small times \cite{beauchard-coron-teismann,beauchard-coron-teismann2} (see also \cite{obstruction-ivan} for different semi-classical obstructions). 

Very recently, the first examples of small-time approximately controllable equations (\ref{eq:schro}) were given in \cite{beauchard-pozzoli} by the authors: they correspond to 
$M=\R^d$, 
multi-input $m=2$, 
$W_1(x)=|x|^2$
and generic $W_2 \in L^{\infty}(\R^d)$. 
The control on the frequency of the quadratic potential $W_1$ permits to construct solutions that evolve approximately along specific diffeomorphisms, namely, space-dilations. Once we have access to space-dilations, we can exploit the scaling of the equation posed on $\R^d$ (with $u_2=0$) to generate time-contractions. In this way, we build on previous results of large-time control, to obtain small-time control. 
However such a strategy cannot be used on equations without scaling, e.g. (\ref{eq:torus}), or equations with continuous spectrum, e.g. \eqref{eq:line}, where previous results of large-time control do not apply.

\medskip

Determining the configurations $(M,V,W_1,\dots,W_m)$ for which approximate controllability of (\ref{eq:schro}) occurs in small time is thus an ambitious open problem. %It requires finding new strategies to prove small-time approximate controllability. 

\medskip

For (nonlinear versions of) systems \eqref{eq:torus} and \eqref{eq:line}, a weaker property was proved:
the small-time controllability of phases \cite{duca-nersesyan,duca-pozzoli}. 
It refers to the possibility, for any initial condition $\psi_0 \in \mathcal{S}$ and phase $\varphi \in L^2(M,\R)$ to reach approximately, in arbitrary small time, the state $\psi_0(x) e^{i \varphi(x)}$. 
%(see Propositions \ref{Prop:DN} and \ref{Prop:DP}). 
The proof relies on Lie bracket techniques and a saturation argument, introduced in the pioneering articles \cite{agrachev-sarychev,agrachev2}, also used in \cite{coron-xiang-zhang,pozzoli,duca-pozzoli-urbani}. The small-time controllability of phases is a consequence of the density in $L^2(M,\R)$ of a particular subspace of ${\rm Lie}\{\Delta-V,W_1,\dots,W_m\}$, 
shared by both systems \eqref{eq:torus} and \eqref{eq:line}. 
For appropriate systems, it implies the small-time approximate controllability between particular eigenstates
\cite{Boscain-2024,chambrion-pozzoli}.

\medskip

In this article, building on the small-time controllability of phases, we propose a new method to prove the small-time $L^2$-approximate controllability of equations of the form \eqref{eq:schro}. We apply it to systems (\ref{eq:torus}) and (\ref{eq:line}) and get the following results.

\begin{theorem}\label{Main_Thm_torus}
If $V \in L^\infty(\T^d,\R)$ then system \eqref{eq:torus} is small-time $L^2$-approximately controllable.
\end{theorem}

\begin{theorem}\label{thm:global-euclidean}
If $V$ satisfies (\ref{Hyp:V_transp}) then system \eqref{eq:line} is small-time $L^2$-approximately controllable.
\end{theorem}

In Section \ref{sec:Strategy}, we develop the proof strategy for Theorems  \ref{Main_Thm_torus} and  \ref{thm:global-euclidean}. The structure of the rest of this article is presented at the end of this section.

%---------------------------------------------------
\section{Proof strategy} \label{sec:Strategy}
\subsection{Some useful notions}

We use \textbf{small-time $L^2$-approximately reachable operators} to describe  states that can be achieved by trajectories of (\ref{eq:schro}) in arbitrarily small time.

\begin{definition}[$L^2$-STAR operators] \label{def:L2STAR}
A unitary operator $L$ on $L^2(M,\C)$ is $L^2$-STAR  if, for every $\psi_0 \in \mathcal{S}$ and $\epsilon>0$, there exist $T \in [0,\epsilon]$, $\theta \in [0,2\pi)$ and $u \in PWC(0,T)$ such that $\| \psi(T;u,\psi_0) - e^{i\theta} L \psi_0 \|_{L^2} < \epsilon$.
\end{definition}

The set of $L^2$-STAR operators forms a subsemigroup closed for the topology of the strong convergence (See Appendix \ref{subsec:STAR_op} for a proof).

\begin{lemma}\label{lem:reachable-operators}
The composition and the strong limit of $L^2$-STAR operators are $L^2$-STAR operators.
\end{lemma}

\begin{definition} \label{def:Diffc0}
${\rm Diff}^0(M)$ (resp. ${\rm Diff}_c^0(M)$) denotes the group of diffeomorphisms of $M$ which are $C^\infty$, isotopic to the identity (resp. and with compact support i.e. for every $P\in {\rm Diff}_c^0(M)$ there exists $ K\subset M$ compact such that $P={\rm Id}$ on $M\setminus K$).
\end{definition}

\begin{definition}[Unitary action of ${\rm Diff}^0(M)$ on $L^2(M,\C)$] 
\label{def:Diffc0_action}
For $P\in {\rm Diff}^0(M)$, the unitary operator on $L^2(M,\C)$ associated with $P$ is defined by
\begin{equation} \label{Def:LP}
\mathcal{L}_{P}\psi =  |J_P|^{1/2}(\psi\circ P),
\end{equation}
where $|J_P| = \text{det}(DP)$ is the determinant of the Jacobian matrix of $P$. Then $\|\mathcal{L}_{P}\psi\|_{L^2}=\|\psi\|_{L^2}.$
\end{definition}

\begin{definition}
${\rm Vec}(M)$ (resp. ${\rm Vec}_c(M)$) denotes the space of 
globally Lipschitz (resp. compactly supported)
smooth vector fields on $M$.
\end{definition}

\begin{definition}[Flows $\phi_f^s$] \label{def:flow}
For $f \in {\rm Vec}(M)$, $\phi_f^s$
denotes the flow associated with $f$ at time $s$:
for every $x_0 \in M$, $x(s)= 
\phi_f^s(x_0)$ is the solution of
the ODE $\dot{x}(s)=f(x(s))$ associated with the initial condition $x(0)=x_0$.
 \end{definition}

For $f \in {\rm Vec}(M)$ (resp. ${\rm Vec}_c(M)$) then $\phi_f^1 \in {\rm Diff}^0(M)$ (resp. ${\rm Diff}_c^0(M)$).

\begin{definition}
We introduce the following small-time controllability (STC) notions:
\begin{itemize}
\item \textbf{STC of phases}: for every $\varphi \in L^2(M,\R)$, the operator $e^{i\varphi}$ is $L^2$-STAR,
\item \textbf{STC of the group ${\rm Diff}_c^0(M)$}: for every $P\in {\rm Diff}_c^0(M)$, the operator $\mathcal{L}_P$ is $L^2$-STAR. 
\item \textbf{STC of flows}: for every $f \in {\rm Vec}_c(M)$, the operator $\mathcal{L}_{\phi_f^1}$ is $L^2$-STAR.
\end{itemize}
\end{definition}

Clearly, the following implication holds
\begin{equation} \label{diffeo_flow}
\text{STC of the group }  {\rm Diff}_c^0(M)\,
 \qquad \Rightarrow \qquad
\text{STC of flows}.
\end{equation}

%---------------------------------------------------
%\subsection{Proof strategy}
\subsection{Reduction to the control of phases and flows}

The goal of this section is to prove the following implication, for systems (\ref{eq:torus}) and (\ref{eq:line})
\begin{equation} \label{phase+flows->STAC}
\text{STC of phases and flows } 
\qquad \Rightarrow  \qquad
\text{$L^2$-STAC.}    
\end{equation}
The first step consists in proving that
\begin{equation} \label{phase+diffeo->STAC}
\text{STC of phases and the group }  {\rm Diff}_c^0(M) 
\qquad \Rightarrow  \qquad
\text{$L^2$-STAC.}    
\end{equation}
This is a consequence of the transitivity of the action of ${\rm Diff}_c^0(M)$ on densities, proved by Moser \cite{moser} and recalled in Appendix \ref{subsec:Moser}.

\begin{theorem}\label{thm:reduction}
The STC of phases and the group ${\rm Diff}_c^0(M)$ implies
the $L^2$-STAC.
\end{theorem}

\begin{proof}
It suffices to prove the property of Definition \ref{def:STAC} for $(\psi_0,\psi_1)$ belonging to a subset of $\mathcal{S} \times \mathcal{S}$ which is dense for the $\|.\|_{L^2}$-topology, because, for every $T>0$ and $u \in \text{PWC}(0,T)$, the operator $\psi(T;u,.)$ is an isometry of $L^2(M,\mathbb{C})$. We use the dense set $\frak{D}(M)$ given by Lemma \ref{Lem:dense_tore} when $M=\T^d$ and Lemma \ref{Lem:dense_line} when $M=\R^d$. Let $(\psi_0,\psi_1) \in \frak{D}(M)$. Then $\psi_j=\rho_j e^{i \phi_j}$ where $\phi_j \in L^2(M,\R)$ and the $\rho_j$ satisfy the assumptions of Proposition \ref{lem:transitivity-torus} if $M=\T^d$ 
and 
Proposition \ref{Prop:Moser_Rd} if $M=\R^d$.
Thus there exist $P\in {\rm Diff}_c^0(M)$ such that 
$\rho_1=\mathcal{L}_P \rho_0$. Then 
$e^{i \phi_1} \mathcal{L}_P e^{-i \phi_0} \psi_0=
e^{i \phi_1} \mathcal{L}_P \rho_0 =
e^{i \phi_1} \rho_1 =\psi_1$. 
By Lemma \ref{lem:reachable-operators}, the operator $e^{i \phi_1} \mathcal{L}_P e^{-i \phi_0}$ is $L^2$-STAR, thus there exist $T \in [0,\epsilon]$ and $u \in \text{PWC}(0,T)$ such that $\|\psi(T;u,\psi_0)-\psi_1\|_{L^2}<\epsilon$.
\end{proof}

The second step consists in proving the reciprocal of (\ref{diffeo_flow}) i.e.
\begin{equation} \label{flows->diffeo}
\text{STC of flows}
\qquad \Rightarrow \qquad
\text{STC of the group } {\rm Diff}_c^0(M).
\end{equation}
This is an application of a result of Thurston \cite{thurston} on the simplicity of the group ${\rm Diff}_c^0(M)$.

\begin{theorem}\label{thm:reduction-thurston}
The STC of flows implies the STC of the group ${\rm Diff}_c^0(M)$.
\end{theorem}

\begin{proof}
Thurston proved in \cite{thurston} that, if $M$ is a connected manifold, then ${\rm Diff}_c^0(M)$ is a simple group (see also \cite{Herman,mather}). Moreover,
 $$F(M):=\{
\phi_{f_n}^1 \circ \cdots \circ \phi_{f_1}^1
; n\in\N^*, f_1,\dots,f_n \in {\rm Vec}_c(M) \}$$
is a normal subgroup of ${\rm Diff}_c^0(M)$. Indeed, if
$X=\phi_{f_n}^1 \circ \cdots \circ \phi_{f_1}^1 \in F(M)$ and $P\in {\rm Diff}^0_c(M)$ then 
$PXP^{-1}
= \phi_{g_n}^1
\circ\dots\circ 
\phi_{g_1}^1$
where $g_j \in {\rm Vec}_c(M)$ is the pushforward of $f_j$ by $P$
$$g_j(x) := (P \star f_j)(x) = DP(P^{-1}(x))\, f_j(P^{-1}(x)),$$
thus $PXP^{-1} \in F(M)$. Therefore ${\rm Diff}_c^0(M) = F(M)$ i.e. any $P \in {\rm Diff}_c^0(M)$ is of the form $P=\phi_{f_n}^1 \circ \cdots \circ \phi_{f_1}^1$. Then
$\mathcal{L}_P = \mathcal{L}_{\phi_{f_n}^1} \circ \dots \circ \mathcal{L}_{\phi_{f_1}^1}$ is $L^2$-STAR by Lemma \ref{lem:reachable-operators}.
\end{proof}

Finally, gathering (\ref{phase+diffeo->STAC}) and (\ref{flows->diffeo}), we obtain (\ref{phase+flows->STAC}). The arguments used in this section work for systems (\ref{eq:torus}) and (\ref{eq:line}), but also, more generally, for any Schrödinger equation (\ref{eq:schro}), posed either on a closed connected manifold $M$ (e.g. $M=\T^d$), or $M=\R^d$, which is well-posed for piecewise constant controls.

%--------------------------------------------
\subsection{STC of phases implies STC of flows}

The STC of phases was recently established for systems \eqref{eq:torus} and \eqref{eq:line}, see Section \ref{sec:phase} for details.

\begin{theorem}[\cite{duca-nersesyan,duca-pozzoli}]\label{thm:DN-DP}
The STC of phases holds for systems (\ref{eq:torus}) and (\ref{eq:line}).
\end{theorem}

By (\ref{phase+flows->STAC}), in order to get Theorems \ref{Main_Thm_torus} and \ref{thm:global-euclidean}, it suffices to prove the following result.

\begin{theorem}\label{thm:vector-control}
The STC of flows holds for systems (\ref{eq:torus}) and (\ref{eq:line}).
\end{theorem}

Our strategy to prove this result is the following one
\begin{equation} \label{strategy}
\text{STC of phases}
\quad \Rightarrow \quad
\text{STC of flows of gradient vector fields}
\quad \Rightarrow \quad
\text{STC of flows.}
\end{equation}
Precisely, we introduce the following space of gradient vector fields on $M$
\begin{equation} \label{def:G_gothique}
    \frak{G}:=\{\nabla \varphi\mid \varphi\in C^\infty\cap L^2(M,\R), \nabla\varphi \in W^{1,\infty}\cap L^4(M) \text{ or }  \nabla\varphi=\text{const.} \}
\end{equation}
(which is just $\{ \nabla \varphi ; \varphi \in C^{\infty}(M,\R)\}$ when $M=\T^d$)
and the first step of our strategy consists in deducing from Theorem \ref{thm:DN-DP} the following result.

\begin{theorem}\label{prop:gradient-control}
Let $f =\nabla \varphi \in \frak{G}$ and $P:=\phi_{f}^1$. Then 
$\mathcal{L}_P$ is $L^2$-STAR for \eqref{eq:torus} and \eqref{eq:line}.
 \end{theorem}

Thanks to Lemma \ref{lem:reachable-operators}, in order to prove Theorem \ref{prop:gradient-control},
it suffices to prove that $\mathcal{L}_P$ is the strong limit of operators that are $L^2$-reachable in arbitrary small time. We use the following convergences
\footnote{The left hand side is $L^2$-approximately reachable in time $\tau^+$ by Theorem \ref{thm:DN-DP} and (an adaptation of) Lemma \ref{lem:reachable-operators}. The Trotter-Kato formula is used in the first convergence.}:
\begin{equation} \label{strong_limits}
\left( 
e^{i \frac{|\nabla \varphi|^2}{4 n\tau}  }
e^{i  \frac{\varphi}{2 \tau} }
e^{i\frac{\tau}{n}(\Delta-V)}
e^{-i \frac{\varphi}{2 \tau} }
\right)^{n}
\quad \underset{n \to \infty}{\longrightarrow} \quad
e^{i \tau (\Delta-V) + \mathcal{T}_f}
\quad \underset{\tau \to 0}{\longrightarrow} \quad
e^{ \mathcal{T}_f }
= \mathcal{L}_{P}
\end{equation}
where
$\mathcal{T}_f = \langle f , \nabla \cdot \rangle + \frac{1}{2} {\rm div}(f)$. A key point is the formulation of $\mathcal{L}_P$ thanks to the group generated by the transport operator $\mathcal{T}_f$ (characteristics method). See Section \ref{sec:gradient} for details.

\bigskip

The second step of our strategy consists in proving that the set of vector fields whose flow define $L^2$-STAR operators has a Lie algebra structure, thus it contains $\text{Lie}(\frak{G})$.

\begin{theorem}\label{prop:lie-algebra}
$\mathfrak{L}:=\{f\in {\rm Vec}(M); \forall t\in \R, \mathcal{L}_{\phi_f^t} \text{ is } L^2\text{-STAR}\}$
is a Lie subalgebra of ${\rm Vec}(M)$.
\end{theorem}

Thanks to Lemma \ref{lem:reachable-operators}, 
in order to prove Theorem \ref{prop:lie-algebra},
it suffices to prove that, if $f,g \in \frak{L}$ and $P:=\phi_{[f,g]}^1$ then $\mathcal{L}_{P}$ is the strong limit of $L^2$-STAR operators. We use the following convergences \footnote{The left hand side is $L^2$-STAR by Lemma \ref{lem:reachable-operators}. The Trotter Kato formula is used in the first convergence.}
\begin{equation}
\left(e^{\frac{-1}{tn}\mathcal{T}_f}e^{-t\mathcal{T}_g}  e^{\frac{1}{tn}\mathcal{T}_f} e^{t\mathcal{T}_g}  \right)^n
\quad \underset{n \to \infty}{\longrightarrow } \quad 
\exp\left(\frac{-1}{t}\mathcal{T}_f+ e^{-t\mathcal{T}_g}  \frac{1}{t}\mathcal{T}_f e^{t\mathcal{T}_g} \right)
\quad \underset{t \to 0}{\longrightarrow } \quad
e^{\mathcal{T}_{[f,g]}} = \mathcal{L}_P,
\end{equation}
see Section \ref{sec:lie-algebra} for details.

\bigskip

The third and last step of our strategy consists in proving that $\text{Lie}(\frak{G})$ is dense in an appropriate sense.

\begin{theorem}\label{thm:gradient-algebra}
Let $M=\T^d$ or $\R^d$. For every $f \in {\rm Vec}_c(M)$, there exists $(f_n)_{n\in\N} \subset \text{Lie}(\frak{G})$ such that $\mathcal{L}_{\phi_f^1}$ is the strong limit of $(\mathcal{L}_{\phi_{f_n}^1})_{n\in\N}$.
\end{theorem}

The proof consists in building explicit elements of $\text{Lie}(\frak{G})$, such as trigonometric functions on $M=\T^d$ or Hermite functions on $M=\R^d$, which are known to be dense in appropriate functional spaces, see Section \ref{sec:gradient-algebra} for details.

Then, by Lemma \ref{lem:reachable-operators}, we conclude that 
${\rm Vec}_c(M) \subset \frak{L}$, which is a reformulation of the control of flows. Thus, Theorem \ref{thm:vector-control} is a consequence of Theorems \ref{prop:gradient-control}, \ref{prop:lie-algebra} and \ref{thm:gradient-algebra}. 

Finally, in order to get Theorems \ref{Main_Thm_torus} and \ref{thm:global-euclidean}, it suffices to prove Theorems \ref{thm:DN-DP}, \ref{prop:gradient-control}, \ref{prop:lie-algebra} and \ref{thm:gradient-algebra}.

\bigskip

In Section \ref{sec:preliminaries}, we recall functional analytic tools extensively used in this article.
In Section \ref{sec:phase} (resp. \ref{sec:gradient}, resp. \ref{sec:lie-algebra}, resp. \ref{sec:gradient-algebra})
we prove Theorem \ref{thm:DN-DP} (resp. \ref{prop:gradient-control}, resp. \ref{prop:lie-algebra}, resp. \ref{thm:gradient-algebra}).

%%%%%%%%%%%%%%%%%%%%%%%%%%%%%%%%%%%%%%%%%%%%%%%%%%%
\section{Functional analytic tools}\label{sec:preliminaries}

In this section, we recall three classical tools of functional analysis that we use extensively in this article.
The first one allows to study the convergence of the propagators at the level of the generators, the second one is the Trotter-Kato formula and the last one is a conjugation formula proved in \cite[Proposition 11]{beauchard-pozzoli}.

\begin{proposition}{\cite[Theorem VIII.21 \& Theorem VIII.25(a)]{rs1}}\label{prop:trotter}
Let $(A_n)_{n\in\N},A$ be self-adjoint operators on a Hilbert space $\mathcal{H}$, with a common core $D$. If $\| (A_n-A)\psi \|_{\mathcal{H}} \underset{n \to\infty}{\longrightarrow} 0$ for any $\psi\in D$, then 
$\| (e^{iA_n}-e^{iA})\psi \|_{\mathcal{H}}\underset{n \to\infty}{\longrightarrow} 0$ for any $\psi\in \mathcal{H}$.
\end{proposition}

\begin{proposition}{(Trotter-Kato product formula) \cite[Theorem VIII.31]{rs1}}\label{prop:trotter-kato}
%{\color{blue}EP: should be just cite this from Reed-Simon book without recalling it?}
Let $A,B$ be self-adjoint operators on a Hilbert space $\mathcal{H}$ such that $A+B$ is essentially self-adjoint on $D(A)\cap D(B)$. Then, for every $\psi_0 \in \mathcal{H}$,
$$ \left\| \left(e^{i\frac{A}{n}}e^{i\frac{B}{n}}\right)^n\psi_0 - e^{i(A+B)}\psi_0\right\|_{\mathcal{H}} \underset{n \to + \infty}{\longrightarrow} 0.$$
\end{proposition}

\begin{proposition}\label{lem:conjugation}
Let $A,B$ be self-adjoint operators on a Hilbert space $\mathcal{H}$, and suppose that $e^{iB}$ is an isomorphism of $D$, where $D$ is a core for $A$. Then, for any $t\in \R$,
$$e^{-iB}e^{itA}e^{iB} =\exp(e^{-iB}itAe^{iB}).$$
\end{proposition}

%%%%%%%%%%%%%%%%%%%%%%%%%%%%%%%%%%%%%%%%%%%%%%%%%%%%%%%%%%%%%%%%%%%
\section{Control of phases} \label{sec:phase}

In this section, we write a proof of Theorem \ref{thm:DN-DP}.
For system \eqref{eq:torus} on  $\T^d$ it is borrowed from  \cite[Theorem A]{duca-nersesyan} and given for sake of completeness. For system \eqref{eq:line} on $\R^d$, it is an adaptation to our framework of the one of \cite[Theorem 1]{duca-pozzoli} for another system.

%---------------------------------------------------------------
\subsection{On $\T^d$}

\begin{proposition} \label{Prop:DN}
Let $V \in L^\infty(\T^d,\R)$.
System \eqref{eq:torus} satisfies the following property:
for every $\varphi \in L^2(\T^d,\R)$,
the operator
$e^{i \varphi}$ is $L^2$-STAR.
\end{proposition}
\begin{proof}
\noindent \emph{Step 1: We prove that, for every $\varphi \in \mathcal{H}_0:=\text{span}_{\R}\{\sin \langle b_j,x\rangle, \cos \langle b_j , x \rangle; j \in \{1,\dots,d\} \}$, the operator $e^{i \varphi}$ is $L^2$-STAR.}
Let $\alpha \in \R^{2d}$ and 
$\varphi: x \in \T^d \mapsto \sum_{j=1}^{d} (\alpha_{2j-1} \sin + \alpha_{2j} \cos) \langle b_j,x\rangle$. For any $\tau>0$, the operator
$L_{\tau}:=e^{ i\tau(\Delta-V) + i\varphi }$
is $L^2$-exactly reachable in time $\tau$,
because associated with the constant controls $u_j(t)=-\alpha_j/\tau$.
For $\tau>0$, the operator $A_{\tau}:=\tau(\Delta-V)+ \varphi$ is self-adjoint on $D(A_{\tau}):=H^2(\T^d,\C)$,  because $V, \varphi \in L^\infty(\T^d,\R)$.
The multiplicative operator $A_0:=\varphi$ is self-adjoint on $L^2(\T^d,\C)$.
$H^2(\T^d,\C)$ is a common core of $A_{\tau}$ and $A_0$. For every $\psi \in H^2(\T^d,\C)$,
$\| (A_{\tau}-A_0)\psi \|_{L^2} =
\tau \|(\Delta-V)\psi\|_{L^2} \rightarrow 0$ as $\tau \to 0$.
Thus, by Proposition \ref{prop:trotter}, for every $\psi \in L^2(\T^d,\C)$,
$\| (L_{\tau}-e^{i\varphi})\psi\|_{L^2}
= \| (e^{i A_{\tau}} - e^{i A_0})\psi \|_{L^2}
\to 0$ as $\tau \to 0$.
By Lemma \ref{lem:reachable-operators}, this proves that 
$e^{i \varphi}$ is $L^2$-STAR.

\medskip

\noindent \emph{Step 2: We prove that, if $\varphi \in C^2(\T^d,\R)$ and $e^{i \lambda \varphi}$ is $L^2$-STAR for every $\lambda \in \R$ then $e^{-i|\nabla\varphi|^2}$ is $L^2$-STAR.}
Let $\tau>0$. By Lemma \ref{lem:reachable-operators}, the operator
$$\widetilde{L}_{\tau}:=
e^{i \frac{\varphi}{\sqrt{\tau}}} 
e^{i\tau(\Delta-V)}
e^{-i \frac{\varphi}{\sqrt{\tau}}}$$
is $L^2$-approximately reachable in time $\tau^+$.

The operator $\tau(\Delta-V)$ is self-adjoint on $H^2(\T^d,\C)$.
The multiplicative operator $\varphi/\sqrt{\tau}$ is self-adjoint on $L^2(\T^d,\C)$. $e^{i\varphi/\sqrt{\tau}}$ is an isomorphism of $H^2(\T^d,\C)$ because $\varphi \in C^2(\T^d,\R)$.
Thus, by Lemma \ref{lem:conjugation} and standard computations
$$\widetilde{L}_{\tau} = \exp \left( i\tau
e^{i \frac{\varphi}{\sqrt{\tau}}} 
(\Delta-V)
e^{-i \frac{\varphi}{\sqrt{\tau}}}  \right)
= \exp \left( i \tau (\Delta-V)-\sqrt{\tau}(2\langle \nabla \varphi , \nabla \rangle + \Delta \varphi) - i |\nabla \varphi|^2 \right).$$
The operator
$$\widetilde{A}_{\tau} := 
\tau
e^{i \frac{\varphi}{\sqrt{\tau}}} 
(\Delta-V)
e^{-i \frac{\varphi}{\sqrt{\tau}}}
=
\tau (\Delta-V)+i\sqrt{\tau}(2\langle \nabla \varphi , \nabla \rangle + \Delta \varphi) - |\nabla \varphi|^2$$
is self-adjoint on $H^2(\T^d,\C)$.
The multiplicative operator 
$\widetilde{A}_0:=- |\nabla \varphi|^2$ 
is self-adjoint on $L^2(\T^d,\C)$ because
$\nabla \varphi \in L^{\infty}(\T^d,\R)$.
$H^2(\T^d,\C)$ is a common core of $\widetilde{A}_{\tau}$ and $\widetilde{A}_0$.
For every $\psi \in H^2(\T^d,\C)$,
$$ \|(\widetilde{A}_{\tau}-\widetilde{A}_0) \psi \|_{L^2} = 
\|\tau (\Delta-V) \psi +i\sqrt{\tau}(2\langle \nabla \varphi , \nabla \rangle  + \Delta \varphi) \psi \|_{L^2} 
\underset{\tau \to 0 }{\longrightarrow} 0.$$
Thus, by Proposition \ref{prop:trotter}, for every $\psi \in L^2(\T^d,\C)$,
$$\| (\widetilde{L}_{\tau} - e^{-i|\nabla \varphi|^2})\psi \|_{L^2} = 
\| (e^{i \widetilde{A}_{\tau}} - e^{i A_0}) \psi \|_{L^2}
\underset{\tau \to 0 }{\longrightarrow} 0.$$
By Lemma \ref{lem:reachable-operators}, this proves that $e^{-i|\nabla \varphi|^2}$ is $L^2$-STAR.

\medskip

\noindent \emph{Step 3: Iteration.} We define by induction an increasing sequence of sets $(\mathcal{H}_{j})_{j\in\N}$: 
$\mathcal{H}_0$ is defined in Step 1 and, for every $j \in \N^*$,
$\mathcal{H}_j$ is the largest vector space whose elements can be written as 
$$  \varphi_0-\sum_{k=1}^N|\nabla\varphi_k|^2 ; N\in\N, \varphi_0,\dots,\varphi_N\in \mathcal{H}_{j-1}. $$
Let $\mathcal{H}_{\infty} := \cup_{j\in\N} \mathcal{H}_j$.
Thanks to Lemma \ref{lem:reachable-operators}, Steps 1 and 2,
for every $\varphi \in \mathcal{H}_{\infty}$, the operator $e^{i \varphi}$ is $L^2$-STAR. Moreover, the proof of \cite[Proposition 2.6]{duca-nersesyan} shows that $\mathcal{H}_{\infty}$ contains any trigonometric polynomial, in particular, $\mathcal{H}_{\infty}$ is dense in $L^2(\T^d,\R)$.

\medskip

\noindent \emph{Step 4: Conclusion.} Let $\varphi \in L^2(\T^d,\R)$. There exists $(\varphi_n)_{n\in\N} \subset \mathcal{H}_{\infty}$ such that $\| \varphi_n -\varphi \|_{L^2} \rightarrow 0 $ as $n \to \infty$. Up to an extraction, one may assume that $\varphi_n \rightarrow  \varphi$ almost everywhere on $\T^d$, as $n \to \infty$. The dominated convergence theorem proves that, for every $\psi \in L^2(\T^d,\C)$,
$\| (e^{i \varphi_n}-e^{i\varphi}) \psi \|_{L^2} \rightarrow 0$ as $n \to \infty$. Finally, Step 3 and Lemma \ref{lem:reachable-operators} prove that the operator $e^{i\varphi}$ is $L^2$-STAR.
\end{proof}

\subsection{On $\R^d$}
\begin{proposition} \label{Prop:DP}
Let $V$ satisfying (\ref{Hyp:V_transp}).
System \eqref{eq:line} satisfies the following properties:
\begin{itemize}
\item for every $j \in \{1,\dots,d\}$ and $u \in \R$, the operator $e^{u \partial_{x_j}}$ is $L^2$-STAR,
\item for every $\varphi \in L^2 (\R^d,\R)$, the operator
$e^{i \varphi}$ is $L^2$-STAR
\end{itemize}
\end{proposition}

\begin{proof}
\noindent \emph{Step 1: We prove that, for every $\varphi \in
\text{span}\{ x_1,\dots,x_d, e^{-|x|^2/2} \}$, the operator $e^{i \varphi}$ is $L^2$-STAR.} Let $\alpha \in\R^{d+1}$ and $\varphi: x \in \R^d \mapsto \alpha_1 x_1 + \dots + \alpha_d x_d + \alpha_{d+1} e^{-|x|^2/2}$. For any $\tau>0$, the operator 
$e^{i\tau(\Delta-V)+i\varphi}$ is $L^2$-exactly reachable in time $\tau$ because associated with the constant controls $u_j=-\alpha_j/\tau$. The operator $\tau(\Delta-V)+\varphi$ is essentially self-adjoint on $C^\infty_c(\R^d,\C)$ thus its closure $A_{\tau}$ is self-adjoint. The multiplicative operator $\varphi$ is self-adjoint on $\{ \psi \in L^2(\R^d,\C) ; \varphi \psi \in L^2(\R^d,\C) \}$. $C^\infty_c(\R^d,\C)$ is a common core of $A_{\tau}$ and $\varphi$. For every $\psi \in C^\infty_c(\R^d,\C)$, 
$\| (A_{\tau}-\varphi) \psi \|_{L^2} =
\tau \| (\Delta-V) \psi \|_{L^2} \to 0$ as $\tau \to 0$. Thus, by Proposition \ref{prop:trotter}, for every $\psi \in L^2(\R^d,\C)$,
$\|(e^{i\tau(\Delta-V)+i\varphi} - e^{i\varphi})\psi  \|_{L^2} \to 0$ as $\tau \to 0$. Therefore, by Lemma \ref{lem:reachable-operators}, the operator $e^{i \varphi}$ is $L^2$-STAR.

\medskip

\noindent \emph{Step 2: We prove that, for every $j \in \{1,\dots,d\}$ and $u \in \R$, the operator $e^{u \partial_{x_j}}$ is $L^2$-STAR.} Let 
$j \in \{1,\dots,d\}$, $u \in \R^*$. By Step 1 and Lemma \ref{lem:reachable-operators}, for every $\tau>0$, the operator
$$L_{\tau}:= e^{\frac{iux_j}{2\tau}}e^{i\tau(\Delta-V)}e^{-\frac{iux_j}{2\tau}} $$
is $L^2$-approximately reachable in time $\tau^+$ .
The operator $\tau(\Delta-V)$ is essentially self-adjoint on $C^\infty_c(\R^d,\C)$ by Proposition \ref{prop:self-adjointness} thus its closure $A$ is self-adjoint. The operator 
$B:=u x_j / (2\tau)$ is self-adjoint on $D(B):=\{ \psi \in L^2(\R^d,\C) ; x_j \psi \in L^2(\R^d,\C) \}$. 
$C^{\infty}_c(\R^d,\C)$ is a core of $A$.
The operator $e^{iB}$ is an isomorphism of $C^{\infty}_c(\R^d,\C)$. Thus, by Proposition \ref{lem:conjugation} and standard computations
\begin{equation}\label{eq:conjugate-dynamics}
L_{\tau} = \exp\left( i\tau
e^{\frac{iux_j}{2\tau}}
(\Delta-V)
e^{-\frac{iux_j}{2\tau}}
\right) = \exp \left(
i \tau (\Delta-V)+u \partial_{x_j} - i \frac{u^2}{2\tau}
\right).
\end{equation}
By Definition \ref{def:L2STAR}, 
$$L_{\tau}' := e^{i \frac{u^2}{2\tau}} L_{\tau}
=  \exp\left( 
i \tau (\Delta-V)+u \partial_{x_j}
\right)$$
is also $L^2$-approximately reachable in time $\tau^+$.
The operator 
$$\tau
e^{\frac{iux_j}{2\tau}}
(\Delta-V)
e^{-\frac{iux_j}{2\tau}}
+ \frac{u^2}{2\tau}
=
\tau (\Delta-V) - i u \partial_{x_j} 
$$
is essentially self-adjoint on $C^\infty_c(\R^d,\C)$ by Proposition \ref{prop:self-adjointness}, thus its closure $A_{\tau}$ is self-adjoint. The operator
$A_0:= -i u \partial_{x_j}$ is self-adjoint on
$D(A_0):=\{ \psi \in L^2(\R^d,\C) ; \partial_{x_j} \psi \in L^2(\R^d,\C) \}$. $C^\infty_c(\R^d;\C)$ is a common core of $A_{\tau}$ and $A_0$. For every $\psi \in C^{\infty}_c(\R^d,\C)$,
$\| (A_{\tau}-A_0)\psi \|_{L^2} = 
\tau \| (\Delta-V) \psi \|_{L^2} \to 0$ as $\tau \to 0$.
Thus, by Proposition \ref{prop:trotter}, for every $\psi \in L^2(\R^d,\C)$, 
$\|(L_{\tau}'-e^{u \partial_{x_j}})\psi \|_{L^2}
=
\| (e^{i A_{\tau}} - e^{i A_0}) \psi \|_{L^2}
\to 0$ as $\tau \to 0$. By Lemma \ref{lem:reachable-operators}, this proves that the operator $e^{u \partial_{x_j}}$ is $L^2$-STAR.

\medskip

\noindent \emph{Step 3: We prove that, if $\varphi \in C^1(\R^d,\R)$ and $e^{i \lambda \varphi}$ is $L^2$-STAR for every $\lambda \in \R$ then the operator $e^{-i \partial_{x_j} \varphi}$ is $L^2$-STAR.} 
Let $\tau>0$. The assumption on $\varphi$, Step 2 and Lemma \ref{lem:reachable-operators} prove that the operator
$$\widetilde{L}_{\tau} := e^{i\frac{\varphi}{\tau}}e^{\tau\partial_{x_j}}e^{-i\frac{\varphi}{\tau}}$$
is $L^2$-STAR. The characteristic method proves that, for every $\psi \in L^2(\R^d,\C)$,
$$\widetilde{L}_{\tau} \psi  : x \in \R^d \mapsto \psi(x+\tau e_j) e^{-i \frac{\varphi(x+\tau e_j)-\varphi(x)}{\tau} }
 \in \C.$$
The continuity of the translation on $L^2(\R^d,\C)$ and the dominated convergence theorem prove that, for every $\psi \in L^2(\R^d,\C)$,
$\| (\widetilde{L}_{\tau}-e^{-i\partial_{x_j}\varphi}) \psi \|_{L^2} \to 0$ as $\tau \to 0$. Finally, by Lemma \ref{lem:reachable-operators}, 
the operator $e^{-i \partial_{x_j} \varphi}$ is $L^2$-STAR.

\medskip

\noindent \emph{Step 4: Iteration.} We define by induction an increasing sequence of sets $(\mathcal{H}_{j})_{j\in\N}$ by
$\mathcal{H}_0=\text{span}\{ e^{-|x|^2/2} \}$ and, for every $j \in \N^*$,
$$\mathcal{H}_j := \text{span}_{\R} \left\{  \varphi_0-\sum_{k=1}^d \partial_{x_k} \varphi_k ;  \varphi_0,\dots,\varphi_d \in \mathcal{H}_{j-1} \right\}$$
and $\mathcal{H}_{\infty} := \cup_{j\in\N} \mathcal{H}_j$.
Thanks to Lemma \ref{lem:reachable-operators}, Steps 1 and 3,
for every $\varphi \in \mathcal{H}_{\infty}$, the operator $e^{i \varphi}$ is $L^2$-STAR. Moreover, by the proof of \cite[Lemma 5.2]{duca-pozzoli}, $\mathcal{H}_{\infty}$ is dense in $L^2(\R^d,\C)$ because it contains the linear combinations of Hermite functions.

\medskip

\noindent \emph{Step 5: Conclusion.} Let $\varphi \in L^2 (\R^d,\R)$. There exists $(\varphi_n)_{n\in\N} \subset \mathcal{H}_{\infty}$ such that $\| \varphi_n -\varphi \|_{L^2} \rightarrow 0 $ as $n \to \infty$. Up to an extraction, one may assume that $\varphi_n \rightarrow  \varphi$ almost everywhere on $\R^d$, as $n \to \infty$. The dominated convergence theorem proves that, for every $\psi \in L^2(\R^d,\C)$,
$\| (e^{i \varphi_n}-e^{i\varphi}) \psi \|_{L^2} \rightarrow 0$ as $n \to \infty$. Finally, Step 4 and Lemma \ref{lem:reachable-operators} prove that the operator $e^{i\varphi}$ is $L^2$-STAR.

\end{proof}

%%%%%%%%%%%%%%%%%%%%%%%%%%%%%%%%%%%%%%%%%%%%%%%%%%%%%%%%%%%%%%%
 \section{Control of flows of gradient vector fields}\label{sec:gradient}

 In this section, we obtain the STC of flows of gradient vector fields (Theorem \ref{prop:gradient-control}) as a corollary of the STC of phases (Theorem \ref{thm:DN-DP}).
We first derive the expression of the group of unitary operators $\{\mathcal{L}_{\phi_f^t}\}_{t\in \R}$ as exponentials of skew-adjoint operators.

\begin{definition}[Unitary transport operator associated with a vector field]
For $f \in {\rm Vec}(M)$, 
the $L^2$-unitary transport operator associated with $f$  is defined by
\begin{equation} \label{Def:T_f}
D(\mathcal{T}_f):=\{ \psi \in L^2(M,\C) ; \langle f,  \nabla \psi\rangle \in L^2(M,\C) \}, \qquad
\mathcal{T}_f \psi = \langle f , \nabla \psi \rangle + \frac{1}{2} {\rm div}(f) \psi.
\end{equation}
\end{definition}

In the first expression, $\nabla \psi$ denotes the  distributional derivative in $\mathcal{D}'(M)$. Since $f$ is globally Lipschitz, $\mathcal{T}_f$ is skew adjoint (we refer to \cite{chernoff} for additional details), thus the group $(e^{t\mathcal{T}_f})_{t\in\R}$ is well defined on $L^2(\T^d,\C)$ as showed in the next lemma. 

\begin{lemma} \label{Lem:op_transp_Tore}
Let $f \in {\rm Vec}(M)$ and 
$P:=
%\mathfrak{e}^{f}
\phi_f^1$.
Then the unitary composition operator associated with $P$
(see (\ref{Def:LP})) satisfies
$\mathcal{L}_P = e^{\mathcal{T}_f}$. 
Moreover for every $t \in \R$ and $\psi \in L^2(M,\C)$,
\begin{equation} \label{SG_transp}
\left( e^{t \mathcal{T}_f} \psi \right)(x)= 
\psi( 
\phi_f^t(x)
%\mathfrak{e}^{tf} x 
) e^{\frac{1}{2} \int_0^t {\rm div} f( 
\phi_f^s(x)
%\frak{e}^{sf} x
) ds}.
\end{equation}
\end{lemma}

\begin{proof} 
    Let $\psi \in L^2(M,\C)$. Then 
    $e^{t \mathcal{T}_f} \psi =\xi(t,.)$ where $\xi$ is the solution of the transport equation
    $$\left\lbrace \begin{array}{ll}
\partial_t \xi(t,x) = \langle f(x),\nabla_x \xi(t,x)\rangle 
+ \frac{1}{2} ({\rm div} f)(x) \xi(t,x), & (t,x) \in \R \times M, \\
\xi(0,.)=\psi.
    \end{array}\right.$$
    The characteristic method proves that $\xi(t,x)$ is given by the right hand side of (\ref{SG_transp}).
    
By definition of $P$, for every $x \in M$,
$dP(x)=R(1)$ where the resolvent $R \in C^0(\R,\mathcal{M}_d(\R))$ solves the linear equation
$\dot{R}(t)= Df( \phi_f^t(x)) R(t)$,
with initial condition
$R(0)=I_d$.
By Liouville formula
$$J_P(x)=\text{det}( dP(x) )= \text{det}(R(1))=
e^{\int_0^1 {\rm tr}D f (
\phi_f^s(x)
%\frak{e}^{sf}(x)) 
ds} = 
e^{\int_0^1 {\rm div } f (
\phi_f^s(x)
%\frak{e}^{sf}(x)
) ds}$$
therefore, for every $\psi \in L^2(M,\C)$,
$$(\mathcal{L}_P \psi)(x) =
|J_P(x)|^{1/2} \psi \circ P(x) =
e^{\frac{1}{2} \int_0^1 {\rm div } f (
\phi_f^s(x)
%\frak{e}^{sf}(x)
) ds}
\psi( \phi_f^1(x) )
= (e^{\mathcal{T}_f} \psi)(x).$$
\end{proof}

\medskip

\begin{proof}[Proof of Theorem \ref{prop:gradient-control}.]
Let $\varphi  \in \frak{G}$ and $f:=2 \nabla \varphi$.
Let $n \in \N^*$ and $\tau>0$.
Theorem \ref{thm:DN-DP} and Lemma \ref{lem:reachable-operators} prove that the operator
$$
L_{\tau,n}:= 
\left( 
e^{i \frac{|\nabla \varphi|^2}{n\tau}  }
e^{i  \frac{\varphi}{\tau} }
e^{i\frac{\tau}{n}(\Delta-V)}
e^{-i \frac{\varphi}{\tau} }
\right)^{n}
$$
is $L^2$-approximately reachable in time $\tau^+$, because
$\varphi, |\nabla\varphi|^2 \in L^2 (M,\R)$.

\medskip

\noindent \emph{Step 1: We prove that
\begin{equation} \label{Step1/Ltaun}
L_{\tau,n} =
\left( 
e^{ i \frac{|\nabla \varphi|^2}{n \tau }  }
\exp \left( i\frac{\tau}{n}\,
e^{i \frac{\varphi}{\tau}  }
(\Delta-V)
e^{-i \frac{\varphi}{\tau} }
\right)
\right)^{n}.
\end{equation}}
The operator $(\Delta-V)$ is essentially self-adjoint on $C^{\infty}_c(M,\C)$ by hypothesis, thus its closure $A$ is self-adjoint. The operator $B:=\varphi/\tau$ is self-adjoint on $D(B):=\{ \psi \in L^2(M,\C) ; \varphi \psi \in L^2(M,\C) \}$.  The operator $e^{i B}$ is an isomorphism of $C^{\infty}_c(M,\R)$ because $\varphi \in C^{\infty}(M)$. 
Thus, Proposition \ref{lem:conjugation}  proves (\ref{Step1/Ltaun}).

\medskip

\noindent \emph{Step 2: We prove that $L_{\tau}$ is $L^2$-reachable in time $\tau^+$, where}
$$L_{\tau}:= 
%\exp \left( 
%i \tau (\Delta-V)+2 \langle \nabla \varphi , \nabla \rangle + \Delta %\varphi 
%\right)= 
\exp \left( 
i \tau (\Delta-V) + \mathcal{T}_f
\right). $$
The operator 
$e^{i \frac{\varphi}{\tau}  }
(\Delta-V)
e^{-i \frac{\varphi}{\tau} }$
is essentially self-adjoint on $C^{\infty}_c(M,\C)$ by hypothesis, thus its closure $A_1$ is self-adjoint. The multiplicative operator $B_1:=|\nabla \varphi|^2/\tau$ is self-adjoint on $L^2(M,\C)$ because $|\nabla \varphi|^2 \in L^\infty(M,\R)$. Hence, $A_1+B_1$ is self-adjoint on $D(A_1)$ because $B_1$ is bounded (by Proposition \ref{prop:kato-rellich}). Thus, by Proposition \ref{prop:trotter-kato},
for every $\psi \in L^2(M,\C)$,
$\| (L_{\tau,n}-L_{\tau}')\psi \|_{L^2} \underset{n \to \infty}{\longrightarrow} 0$
where
$$L_{\tau}':=\exp \left( i  \left\{  \frac{|\nabla \varphi|^2}{\tau}  + \tau 
e^{i \frac{\varphi}{\tau}  }
(\Delta-V)
e^{-i \frac{\varphi}{\tau}}
  \right\} \right)
  $$
  and standard computations prove that
\begin{equation} \label{Step2:Log}
      \tau 
e^{i \frac{\varphi}{\tau}  }
(\Delta-V)
e^{-i \frac{\varphi}{\tau}}
+ \frac{|\nabla \varphi|^2}{\tau} 
=
\tau (\Delta-V) - 2 i \langle \nabla \varphi , \nabla \rangle  - i \Delta \varphi 
=
\tau (\Delta-V) - i \mathcal{T}_f
\end{equation}
thus $L_{\tau}'=L_{\tau}$. By Lemma \ref{lem:reachable-operators}, 
the operator $L_{\tau}$ is $L^2$-approximately reachable in time $\tau^+$. 

\medskip

\noindent \emph{Step 3: We prove that the operator 
$e^{\mathcal{T}_f}$ is $L^2$-STAR.}
The operator (\ref{Step2:Log}) defined on $C^{\infty}_c(M,\C)$ has a self-adjoint closure $A_{\tau}$ by hypothesis and Proposition \ref{prop:kato-rellich}.
The operator
$A_0:=- i \mathcal{T}_f$
is self-adjoint because $\nabla \varphi \in W^{1,\infty} (M,\R)$ (see, e.g., \cite{chernoff}).
$C^{\infty}_c(M,\C)$ is a common core of $A_{\tau}$ and $A_0$.
%because
%$\nabla \varphi \in L^{\infty}(\R^d,\R)$.
For every $\psi \in C^\infty_c(M,\C)$,
$\| (A_{\tau}-A_0) \psi \|_{L^2} =
\tau \left\| (\Delta-V)\psi \right\|_{L^2}    
\to 0$ as $\tau \to 0$.
Thus, by Proposition \ref{prop:trotter}, for every $\psi \in L^2(M,\C)$,
$\| 
(L_{\tau} - e^{ \mathcal{T}_f }) \psi
\|_{L^2} 
= \| (e^{i A_{\tau}} - e^{i A_0}) \psi \|_{L^2}
\to 0$ as $\tau \to 0$.
Finally, by Lemma \ref{lem:reachable-operators}, the operator $e^{\mathcal{T}_f}$ is $L^2$-STAR.  
\end{proof}

%%%%%%%%%%%%%%%%%%%%%%%%%%%%%%%%%%%%%%%%%%%%%%%%%%%%%%
\section{A Lie algebra structure }\label{sec:lie-algebra}

In this section, we prove Theorem \ref{prop:lie-algebra}.

\begin{proof}[Proof of Theorem \ref{prop:lie-algebra}.]
Clearly, if $f \in \frak{L}$ and $\lambda \in \R$ then $\lambda f \in \frak{L}$. Thus, it suffices to prove that $\frak{L}$ is stable by summation and Lie bracket. 

\medskip

\noindent \emph{Step 1: $\frak{L}$ is stable by summation.}
It suffices to prove that 
$( f,g\in \frak{L} \Longrightarrow e^{\mathcal{T}_{f+g}} \text{ is } L^2\text{-STAR } )$
because, for every $t \in \R$, $tf, tg \in \frak{L}$ and $t \mathcal{T}_{f+g} = \mathcal{T}_{tf+tg}$.
Let $f,g\in\frak{L}$.
By Lemma \ref{lem:reachable-operators}, for every $n\in\N^*$,
the operator
$( e^{\frac{1}{n}\mathcal{T}_f}e^{\frac{1}{n}\mathcal{T}_g} )^n$
is $L^2$-STAR. 
The operators $\mathcal{T}_f$ and $\mathcal{T}_g$ are skew adjoint on $L^2(\T^d,\C)$ and the operator $\mathcal{T}_f+\mathcal{T}_g=\mathcal{T}_{f+g}$ is essentially skew adjoint on $D(\mathcal{T}_f) \cap D(\mathcal{T}_g)$. Thus, by Proposition \ref{prop:trotter-kato}, for every $\varphi \in L^2(\mathbb{T}^d,\C)$, 
$\| ( e^{\frac{1}{n}\mathcal{T}_f}e^{\frac{1}{n}\mathcal{T}_g} )^n \varphi - e^{\mathcal{T}_{f+g}} \varphi \|_{L^2} \to 0$ as $n \to \infty$. By Lemma \ref{lem:reachable-operators}, the operator $e^{\mathcal{T}_{f+g}}$ is $L^2$-STAR. 
 
\medskip

\noindent \emph{Step 2: $\frak{L}$ is stable by Lie bracket, i.e. $(f,g\in \frak{L} \Longrightarrow [f,g]:=(Dg)f-(Df)g\in \frak{L}) $}. It suffices to prove that 
$( f,g \in \frak{L} \Longrightarrow e^{\mathcal{T}_{[f,g]}}\text{ is } L^2 \text{-STAR} )$, because, for every $t \in \R$, $t f \in \frak{L}$ and
$t\mathcal{T}_{[f,g]}=\mathcal{T}_{[tf,g]}$.
For $t \in \R^*$ and $n \in \N^*$, the operator 
$$L_{t,n} := \left(e^{\frac{-1}{tn}\mathcal{T}_f}e^{-t\mathcal{T}_g}  e^{\frac{1}{tn}\mathcal{T}_f} e^{t\mathcal{T}_g}  \right)^n$$
is $L^2$-STAR thanks to Lemma \ref{lem:reachable-operators}. The transport operators $\mathcal{T}_f,\mathcal{T}_g$ are skew-adjoint.
By (\ref{SG_transp}), $e^{t\mathcal{T}_g}$ is an isomorphism of $C^\infty_c(M,\C)$.Thus by applying Proposition \ref{lem:conjugation} we get
$$L_{t,n}=\left(e^{\frac{-1}{tn}\mathcal{T}_f}\exp\left(e^{-t\mathcal{T}_g}  \frac{1}{tn}\mathcal{T}_f e^{t\mathcal{T}_g}\right)\right)^n.$$
By Proposition \ref{prop:trotter-kato}, 
$$
\forall \varphi \in L^2(M,\C), \qquad
\| (L_{t,n}-L_{t})\varphi \|_{L^2} \underset{n \to \infty}{\longrightarrow} 0,
$$
where
$$L_t=\exp\left(\frac{-1}{t}\mathcal{T}_f+ e^{-t\mathcal{T}_g}  \frac{1}{t}\mathcal{T}_f e^{t\mathcal{T}_g} \right).$$
By Lemma \ref{lem:reachable-operators}, this proves that $L_t$ is $L^2$-STAR. By combining Lemma \ref{Lem:crochet} below, with Proposition \ref{prop:trotter}, we obtain, for every $\varphi \in L^2(M,\C)$, $\|(L_t-e^{\mathcal{T}_{[f,g]}})\varphi \|_{L^2} \to 0$ as $t \to 0$. Finally, Lemma \ref{lem:reachable-operators} proves that the operator $e^{\mathcal{T}_{[f,g]}}$ is $L^2$-STAR.
%{\color{magenta}Using Lemma \ref{Lem:op_transp_Tore}, we compute
%\begin{align*}
%&e^{-t\mathcal{T}_g}  \frac{1}{t}\mathcal{T}_f %e^{t\mathcal{T}_g}=e^{\frac{1}{2}\int_0^t{\rm div}g(\frak{e}^{(r-%t)g}x)dr+\frac{1}{2}\int_{0}^{-t}{\rm div}g(\frak{e}^{-rg}x)dr}\\
%& \left(D\frak{e}^{tg}(e^{-tg}x)\frac{f(\frak{e}^{-tg}x)}
%{t}\nabla+\langle f(\frak{e}^{-tg}x),\frac{1}{2t}\int_{0}^t D \frak{e}^{rg}(\frak{e}^{-tgx})\nabla {\rm div} g(\frak{e}^{(r-t)g}x)dr\rangle+\frac{{\rm div}f(\frak{e}^{-tg}x)}{2t}\right).
%\end{align*}
%Now, by expanding
%\begin{equation}\label{eq:expansion}
%D\frak{e}^{tg}=I+tDg+o(t), \quad \frac{f(\frak{e}^{-tg}x)}%{t}=\frac{f(x)}{t}-Df(x)g(x)+o(1),
%\end{equation}
%we see that the operator exponentiated in $L_t$ converges pointwise in $L^2$, as $t\to 0$, to
%$((Dg)f-(Df)g)\nabla+\frac{1}{2}(\langle f,\nabla{\rm div}g\rangle-\langle \nabla {\rm div}f,g\rangle)=\mathcal{T}_{[f,g]}$. The conclusion follows from Proposition \ref{prop:trotter}.}
\end{proof}

\begin{lemma} \label{Lem:crochet}
Let $f,g \in {\rm Vec}(M)$ and 
$\psi \in C^{\infty}_c(M,\C)$. Then
$$\left\| \left(
e^{-t\mathcal{T}_g}\, \mathcal{T}_f\, e^{t\mathcal{T}_g}  
- \mathcal{T}_f  - t\, \mathcal{T}_{[f,g]} \right)\psi
\right\|_{L^2} = \underset{t \to 0}{o}(t). $$
\end{lemma}

\begin{proof}
\noindent \emph{Step 1: We prove}
$$(e^{-t\mathcal{T}_g}\, \mathcal{T}_f\, e^{t\mathcal{T}_g} \psi)(x)
 = 
\langle (\phi_g^t \star f) (x) , \nabla \psi (x) \rangle 
 + 
\frac{1}{2} \left( {\rm div} f(\phi_g^{-t}(x)) 
+ t  G(t,x)
  \right) \psi(x) $$
where $(\phi_g^t \star f)$ is the push forward of the vector field $f$ by the diffeomorphism $\phi_g^t$, i.e.
$$(\phi_g^t \star f) (x) := D \phi_g^t( \phi_g^{-t}(x) ) f ( \phi_g^{-t}(x) ) =
(D \phi_g^{-t}(x))^{-1}f ( \phi_g^{-t}(x) )$$
and
$$G(t,x):=\int_{-1}^0 D({\rm div} g)(\phi_g^{t\theta}(x))
D \phi_g^{t \theta}(\phi_g^{-t}(x)) 
(\phi_g^t \star f)(x) d\theta.$$
Using (\ref{SG_transp}) and (\ref{Def:T_f}), we obtain
\begin{equation} \label{interm1}
\begin{aligned}
(e^{-t\mathcal{T}_g}\, \mathcal{T}_f\, e^{t\mathcal{T}_g} \psi)(x)
& =
(\mathcal{T}_f\, e^{t\mathcal{T}_g} \psi)( \phi_g^{-t}(x)) e^{\frac{1}{2}\int_0^{-t} {\rm div} g ( \phi_g^s(x) )ds  }
\\ & = \left( 
\langle f , \nabla (e^{t\mathcal{T}_g} \psi) \rangle + \frac{1}{2} {\rm div}(f) (e^{t\mathcal{T}_g} \psi)\,
\right) (\phi_g^{-t}(x))
    e^{-\frac{1}{2}\int_{-t}^{0} {\rm div} g ( \phi_g^s(x) )ds  }.
\end{aligned}
\end{equation}
We deduce from the expression (\ref{SG_transp}) that
\begin{equation} \label{interm2}
(e^{t\mathcal{T}_g} \psi)(\phi_g^{-t}(x)) =
\psi(x) e^{\frac{1}{2} \int_{-t}^0 {\rm div} g ( \phi_g^s(x) ds}.
\end{equation}
We deduce from (\ref{SG_transp}) and the chain rule that
\begin{align}
& \langle f(x) , \nabla(e^{t \mathcal{T}_g} \psi)(x) \rangle
= D(e^{t \mathcal{T}_g} \psi)(x) f(x) 
\\  = &  
\left( D\psi(\phi_g^t(x)) D \phi_g^t(x) f(x) +
\frac{1}{2} \int_0^t D({\rm div} g )( \phi_g^s(x)) D\phi_g^s(x) f(x) ds\,  \psi(\phi_g^t(x))
\right)e^{\frac{1}{2} \int_0^t {\rm div} g (\phi_g^s(x))ds}.
\end{align}
We deduce from the resolvent relation 
$D\phi_g^s(y) (D\phi_g^t(y))^{-1} =D \phi_g^{s-t}(y)$ that
\begin{equation} \label{interm3}
\begin{aligned}
& \langle f , \nabla(e^{t \mathcal{T}_g} \psi) \rangle (\phi_g^{-t}(x)) =
\Big( D\psi(x) (\phi_g^t \star f)(x) 
\\  + & 
\frac{1}{2} \int_0^t D({\rm div} g )( \phi_g^{s-t}(x)) 
D \phi_g^{s-t}( \phi_g^{-t}(x)) (\phi_g^t \star f)(x) ds\, 
 \psi(x)
\Big)e^{\frac{1}{2} \int_{-t}^0 {\rm div} g (\phi_g^s(x))ds}.
\end{aligned}
\end{equation}
We conclude Step 1 by gathering (\ref{interm1}), (\ref{interm2}) and (\ref{interm3}), in which we use the change of variable $s-t=t\theta$.

\medskip

\noindent \emph{Step 2: Conclusion.}
We deduce from Step 1 and the relation
$$\mathcal{T}_{[f,g]}=\langle [f,g] ,\nabla \rangle + \frac{1}{2} \left( D({\rm div} g)f- D({\rm div} f) g
\right)$$
that
\begin{equation} \label{estim00}
\begin{aligned}
\left( e^{-t\mathcal{T}_g}\, \mathcal{T}_f\, e^{t\mathcal{T}_g}  
- \mathcal{T}_f - t\, \mathcal{T}_{[f,g]} \right) \psi
= &
\left\langle \phi_g^t \star f - f - t [f,g]  , \nabla \psi \right\rangle 
+ \frac{t}{2}
\left( G(t,.)-G(0,.) \right) \psi 
\\ & + \frac{1}{2} \left(
{\rm div} f \circ \phi_g^{-t} - {\rm div} f + t\, D( {\rm div} f ) g 
\right) \psi.
\end{aligned}
\end{equation}
Thus, with $K:=\text{Supp}(\psi)$,
\begin{equation} \label{estim0}
\begin{aligned}
\| \left( e^{-t\mathcal{T}_g}\, \mathcal{T}_f\, e^{t\mathcal{T}_g} 
- \mathcal{T}_f - t\, \mathcal{T}_{[f,g]} \right) \psi \|_{L^2}
\leq 
& 
\| \phi_g^t \star f - f - t [f,g] \|_{L^{\infty}(K)} 
\| \nabla \psi \|_{L^2} 
\\ &
+ \frac{t}{2}
\| G(t,.)-G(0,.) \|_{L^{\infty}(K)}  \| \psi \|_{L^2}
\\ & 
+ \frac{1}{2} 
\| {\rm div} f \circ \phi_g^{-t} - {\rm div} f + t\, D( {\rm div} f ) g 
\|_{L^{\infty}(K)}  \|\psi\|_{L^2}.
\end{aligned}
\end{equation}

The expression $F(t,x):=(\phi_g^t \star f)(x)$ defines  $F \in {\rm Vec}(M)$ thus, by Taylor formula, for every $(t,x) \in \R \times M$,
$$\left| F(t,x)-F(0,x)-t \frac{\partial F}{\partial t}(0,x) \right| \leq
\frac{|t|^2}{2}\, \sup \left\{ \left| \frac{\partial^2 F}{\partial \tau^2}(\tau,y) \right|;
(\tau,y) \in [0,t] \times M  \right\}.$$
Moreover, the chain rule  and the resolvent relation
$$\frac{\partial}{\partial t} D\phi_g^{t}(x) = Dg(\phi_g^t(x))  D\phi_g^{t}(x) $$
prove that
$$\frac{\partial F}{\partial t}(t,x)
=  (D \phi_g^{-t}(x))^{-1} Dg \left( \phi_g^{-t}(x) \right) f( \phi_g^{-t}(x)) 
- (D \phi_g^{-t}(x))^{-1} Df( \phi_g^{-t}(x)) g(\phi_g^{-t}(x))$$
thus
$$\frac{\partial F}{\partial t}(0,x) = [f,g](x).$$
The continuous function $(t,x) \mapsto \partial_t^2 F(t,x)$ is bounded on the compact set $[0,1] \times K$, thus
\begin{equation} \label{estim1}
\| \phi_g^t \star f - f - t [f,g] \|_{L^{\infty}(K)} = \underset{t \to 0}{o}(t).
\end{equation}
The same argument applied to the expression $F(t,x):={\rm div} f( \phi_g^{-t}(x))$ gives
\begin{equation} \label{estim2}
\left\| 
{\rm div} f \circ \phi_g^{-t} - {\rm div} f + t D( {\rm div} f ) g 
\right\|_{L^{\infty}(K)}  = \underset{t \to 0}{o}(t) . 
\end{equation}
The same argument applied to $G$ proves 
\begin{equation} \label{estim3}
\|G(t,.)-G(0)\|_{L^{\infty}(K)}=\underset{t \to 0}{o}(1).
\end{equation}
The estimates  (\ref{estim0}), (\ref{estim1}), (\ref{estim2}),(\ref{estim3}) give the conclusion. 
\end{proof}

%%%%%%%%%%%%%%%%%%%%%%%%%%%%%%%%%%%%%%%%%%%%%%%%%%%%%%%%%
\section{The Lie algebra generated by gradient vector fields}\label{sec:gradient-algebra}

In this section, we prove Theorem \ref{thm:gradient-algebra}.
%It turns out that the Lie algebra generated by gradient vector fields when $M=\T^d$ (and, additionally, the constant vector fields when $M=\R^d$) has density properties  We could not find a reference where such result is proved more in general, hence we furnish the following proof.

 %-------------------------------------------------------
\subsection{On $\T^d$}
\begin{proposition}\label{prop:algebra-torus}
Let $M=\T^d$. Then, $\frak{L}_0:={\rm Lie}(\mathfrak{G})$ contains any linear combination of vector fields of the form
\begin{equation} \label{Monome}
f(x) = \underset{i \in I}{\Pi} \cos( \ell_i x_i)  
\underset{i \in I^c}{\Pi} \sin( \ell_i x_i ) e_k
\end{equation}
where 
$I$ is a subset of $\{1,\dots,d\}$,
$k \in \{1,\dots,d\}$, 
$(\ell_1,\dots \ell_d) \in \N^d$, and $e_k=\partial_{x_k}$. Hence Theorem \ref{thm:gradient-algebra} on $\T^d$ holds.
\end{proposition}

The proof of Proposition \ref{prop:algebra-torus} is inspired by \cite[Lemma 6.5]{agrachev-sarychev3} (we refer also to \cite{agrachev-ensemble-2,agrachev-sarychev2} for similar Lie bracket methods, and to \cite{agrachev-caponigro,trelat-2017} for previous controllability results of the group of diffeomorphisms).

The main difference and difficulty here w.r.t. this literature is that we only have access to \emph{gradient} vector fields.
\begin{proof}
\noindent \emph{Step 1: $\frak{L}_0$ contains any vector field of the form $f(x)=\cos(\ell x_j) e_j$ (resp. $\sin(\ell x_j) e_j$) for $j \in \{1,\dots,d\}$ and $\ell \in \N^*$.} Indeed $f=\nabla \varphi$ where
$\varphi(x)=\sin(\ell x_j)/\ell$ (resp. $\cos(\ell x_j)/\ell$).

\medskip

\noindent \emph{Step 2: $\frak{L}_0$ contains $f(x)= e_j$ for any $j \in \{1,\dots,d\}$.} By Step 1, $\frak{L}_0$ contains
\begin{align}
[\cos(x_j) e_j , \sin(x_j) e_j]
& = 
\cos(x_j) \frac{\partial}{\partial x_j} \left( \sin(x_j) e_j \right) -
\sin(x_j) \frac{\partial}{\partial x_j} \left( \cos( x_j ) e_j \right)
\\ & = (\cos^2(x_j)  +\sin^2(x_j)) e_j = e_j.
\end{align}

\medskip

\noindent \emph{Step 3: $\frak{L}_0$ contains $f(x)=\cos(x_j) e_k$ (resp. $\sin(x_j) e_k$) for any $j \neq k \in \{1,\dots,d\}$.}
Indeed, $\frak{L}_0$ contains
$$-\frac{1}{2}[\nabla \sin(x_j) \sin(x_k) , \nabla \cos(x_k) ] + \frac{1}{2} [\nabla \sin(x_j) \cos( x_k) , \nabla \sin (x_k)] 
= \sin(x_j) e_k$$
$$\text{(resp. } \quad
\frac{1}{2}[\nabla \cos(x_j) \cos(x_k) , \nabla \sin(x_k) ] -\frac{1}{2} [\nabla \cos(x_j) \sin( x_k) , \nabla \cos (x_k)] 
= \cos(x_j) e_k \text{).}
$$
These elementary calculations can be carried out as in Step 2.

\medskip

\noindent \emph{Step 4: We prove by induction on $\ell \in \N$ that
$\frak{L}_0$ contains $f(x)=\cos( \ell x_j) e_k$ (resp. $\sin(\ell x_j) e_k$) for every $j \neq k \in \{1,\dots,d\}$ and $\ell \in \N^*$.} The initialization for $\ell=1$ is given by Step 3. 
We assume the property proved up to $\ell$. 
By Step 1 and the induction assumption, $\frak{L}_0$ contains
$$
\frac{1}{\ell}[\sin(x_j) e_j , \cos(\ell x_j) e_k] + 
\frac{1}{\ell}[\cos(x_j)e_j , \sin(\ell x_j) e_k ]
=
\cos((\ell+1)x_j) e_k,
$$
$$
\frac{1}{\ell}[\sin(x_j) e_j , \sin(\ell x_j) e_k] - 
\frac{1}{\ell}[\cos(x_j)e_j , \cos(\ell x_j) e_k ]
 =
\sin((\ell+1)x_j) e_k.
$$

\medskip

\noindent \emph{Step 5: We prove the first statement of Proposition \ref{prop:algebra-torus} by induction on the degree of $f$, defined by $d(f):=\sharp\{ i \in \{1,\dots,d\} ; \ell_i \neq 0\}$.} The initialization for $d(f)=1$ is given by Step 4. Let $s\in\N^*$. We assume the property proved for any monomial of degree $\leq s$. Let $f$ be a monomial with degree $(s+1)$. Then $f(x)=g(x)\cos( \ell_{\alpha} x_{\alpha}) e_k$ or
$g(x)\sin( \ell_{\alpha} x_{\alpha}) e_k$
where $g$ is a monomial with degree $s$ independent of $x_{\alpha}$.

\medskip

\noindent \emph{First case: $g(x)$ does not depend on $x_k$.} By the induction assumption and Step 4, $\frak{L}_0$ contains
$$\frac{1}{\ell_{\alpha}} [g(x) e_{\alpha} , 
\sin( \ell_{\alpha} x_{\alpha} ) e_k ] = 
g(x) \cos( \ell_{\alpha} x_{\alpha}) e_k,$$
$$-\frac{1}{\ell_{\alpha}} [g(x) e_{\alpha} , 
\cos( \ell_{\alpha} x_{\alpha} ) e_k ] = 
g(x) \sin( \ell_{\alpha} x_{\alpha}) e_k.$$

\medskip

\noindent \emph{Second case: $g(x)$ depends on $x_k$.} Then $\alpha \neq k$. Thus there exists a monomial $g_1(x)$ of degree $s$ such that $g(x)=\partial_{x_k} g_1(x)$. By induction assumption and Step 4, $\frak{L}_0$ contains
$$[\cos(\ell_{\alpha} x_{\alpha}) e_k , g_1(x) e_k] = \cos(\ell_{\alpha}) g(x) e_k,$$
$$[\sin(\ell_{\alpha} x_{\alpha}) e_k , g_1(x) e_k] = \sin(\ell_{\alpha}) g(x) e_k.$$

\medskip

\noindent \emph{Step 6: We prove Theorem \ref{thm:gradient-algebra} for $M=\T^d$.} Let $f \in {\rm Vec}(\T^d)$. By applying Fejer theorem to $f$ and 
$\nabla f$, we obtain a sequence $(f_k)_{k \in \N}$ of trigonometric polynomials vector fields such that $\|f_{k}-f\|_{W^{1,\infty}} \to 0$ as $k \to \infty$. 

%By the first statement of Proposition \ref{prop:algebra-torus} and Proposition \ref{Prop:Lie_alg},
%for every $k\in\N$, the operator $e^{\mathcal{T}_{f_{k}}}$ is $L^2$-STAR. By Lemma \ref{lem:reachable-operators}, to prove that $e^{\mathcal{T}_{f}}$ is $L^2$-STAR, 

It suffices to prove that
\begin{equation} \label{CV}
\forall \psi \in L^2(\T^d,\C), \quad
\|(e^{\mathcal{T}_{f_{k}}} - e^{\mathcal{T}_{f}})\psi \|_{L^2} \underset{\epsilon \to 0}{\longrightarrow} 0.
\end{equation}
It suffices to prove (\ref{CV}) for every $\psi$ in a dense subset of $L^2(\T^d,\C)$ because the operators $(e^{\mathcal{T}_{f_{k}}}-e^{\mathcal{T}_f})$ are bounded on $L^2$ uniformly with respect to $k$. So we consider $\psi \in L^{\infty}(\T^d,\C)$ . Gronwall Lemma proves that
\begin{equation} \label{flot_proches}
\forall t \in \R, \qquad
\| \phi_{f}^t - \phi_{f_{k}}^t \|_{L^{\infty}} \leq t \|f-f_{k} \|_{L^{\infty}} e^{t \|f\|_{C^1}}.
\end{equation}
Thus, for every $x \in \T^d$,
\begin{equation} \label{intdiv_proches}
%\begin{aligned}
%&  
\int_0^1 \left| {\rm div} f (\phi_f^s(x))  - 
 {\rm div} f_{k} (\phi_{f_{k}}^s(x)) 
\right| ds
%\\  \leq &  \int_0^1 
%|{\rm div} f (\phi_f^s(x)) - 
% {\rm div} f (\phi_{f_{k}}^s(x))| + 
%|{\rm div} f (\phi_{f_{k}}^s(x)) -  
% {\rm div} f_{\epsilon} (\phi_{f_{k}}^s(x))| ds
%\\  
\leq 
%&  
\|f\|_{C^2} \|f-f_{k}\|_{L^{\infty}} e^{\|f\|_{C^1}} + \|f-f_{k}\|_{W^{1,\infty}} \underset{k \to \infty}{\longrightarrow} 0.
%\end{aligned}
\end{equation}
For  almost every $x \in \T^d$, we have the convergence 
$e^{\mathcal{T}_{f_{k}}} \psi (x) \to  
e^{\mathcal{T}_{f}}  \psi (x)$ as $k \to \infty$ and 
the domination, for $k$ large enough
$|e^{\mathcal{T}_{f_{k}}} \psi (x)| \leq 
\|\psi\|_{L^{\infty}} e^{\|f\|_{C^1}+1}$.
The dominated convergence theorem proves that 
$\|(e^{\mathcal{T}_{f_{k}}} - e^{\mathcal{T}_{f}})\psi \|_{L^2} \to 0$ as $\epsilon \to 0$.
\end{proof}

%---------------------------------------------------
\subsection{On $\R^d$}

\begin{proposition}\label{prop:algebra-euclidean space}
Let $M=\R^d$. Then, $\frak{L}_0:={\rm Lie}(\frak{G})$ contains any linear combination of vector fields of the form
\begin{equation} \label{Hermite}
\partial_{x_1}^{n_1}\dots\partial_{x_d}^{n_d}\, 
e^{-|x|^2/2}e_j 
\end{equation}
where $n=(n_1,\dots,n_d) \in\N^d$ and $j \in \{1,\dots,d\}$.
Hence Theorem \ref{thm:gradient-algebra} on $\R^d$ holds.
\end{proposition}

\begin{proof} 
\emph{Step 1: We prove the first statement.}
Let $n \in\N^d$ and $j \in \{1,\dots,d\}$.
By definition of $\frak{L}_0$, this Lie algebra contains the following vector fields 
\begin{align*}
& h_j:= -4[\nabla x_j e^{-|x|^2/4},\nabla e^{-|x|^2/4}]=(2+|x|^2)e^{-|x|^2/2}e_j,
\\
& k_j:= -8[\nabla \frac{x^2_j}{2} e^{-|x|^2/4},\nabla x_je^{-|x|^2/4}]=(8-2x^2_j +x^2_j |x|^2)e^{-|x|^2/2}e_j, 
\\
& {\rm ad}_{e_j}^2 h_j=(-2x_j^2-|x|^2+x_j^2|x|^2)e^{-|x|^2/2}e_j,
\\
& \ell_j := k_j-{\rm ad}_{e_j}^2h_j=(8+|x|^2)e^{-|x|^2/2}e_j,
\\
& m_j := (\ell_j-h_j)/6=e^{-|x|^2/2}e_j,
\\
& {\rm ad}_{e_1}^{n_1}\dots {\rm ad}_{e_d}^{n_d} m_j =
\partial_{x_1}^{n_1}\dots\partial_{x_d}^{n_d}e^{-|x|^2/2}e_j.
\end{align*}

\noindent \emph{Step 2: Density of $\frak{L}_0$.} By Step 1,
$$\text{Span}\{ \psi_n(x) e_j\, ;\, n \in \N^d, j \in \{1,\dots,d\} \}
\subset \frak{L}_0, $$
where the $\psi_n$ are the Hermite functions on $\R^d$. Thus $\frak{L}_0$ is dense in $L^2(\R^d)$, but it is also, for any $s\geq 0$, dense in $L^2((1+|x|^2)^s dx)$ and in $H^s(\R^d)$ (by Plancherel, because $\widehat{\psi}_n=(i)^{|n|} \psi_n$).

\medskip

\noindent \emph{Step 3: We prove Theorem \ref{thm:gradient-algebra} for $M=\R^d$.} Let $f \in C^{\infty}_c(\R^d,\R^d)$.
By Step 2, there exists a sequence 
$(f_{k})_{k\in\N} \subset \frak{L}_0$
such that $\|f-f_{k}\|_{W^{1,\infty}} \to 0$ as $k \to \infty$. 

%By Proposition \ref{Prop:Lie_alg-euclidean space}, for every $k\in\N$, the operator $e^{\mathcal{T}_{f_{k}}}$ is $L^2$-STAR. By Lemma \ref{lem:reachable-operators}, to prove that $e^{\mathcal{T}_f}$ is $L^2$-STAR, 

It suffices to prove that
\begin{equation} \label{bla}
\forall \psi \in L^2(\R^d,\C)\,, \qquad 
\| (e^{\mathcal{T}_{f_{k}}}-e^{\mathcal{T}_f})\psi \|_{L^2} \underset{\epsilon \to 0}{\longrightarrow} 0.
\end{equation}
It suffices to prove (\ref{bla}) for any $\psi$ in a dense subset of $L^2$ because the operators $(e^{\mathcal{T}_{f_{k}}}-e^{\mathcal{T}_f})$ are bounded on $L^2$ uniformly with respect to $k$. So we consider $\psi \in C^{\infty}_c(\R^d,\C)$. By (\ref{flot_proches}) and (\ref{intdiv_proches}), for almost every $x \in \R^d$, we have the convergence 
$e^{\mathcal{T}_{f_{k}}} \psi (x) \to  
e^{\mathcal{T}_{f}}  \psi (x)$ as $k \to \infty$ and 
the domination
$|e^{\mathcal{T}_{f_{k}}} \psi (x)| \leq \|\psi\|_{\infty} e^{\|f\|_{\infty}+1} 1_{K}(x)$ where $K$ is a compact subset of $\R^d$ that contains $\cup_{k\in\N} \phi_{f_k}^1(\text{Supp}(\psi))$.
The dominated convergence theorem gives the conclusion.
\end{proof}

 %Hence, thanks to Proposition \ref{prop:lie-algebra}, in order to prove that any operator $\mathcal{L}_P, P\in {\rm Diff}_c^0(M),$ is $L^2$-STAR for systems \eqref{eq:torus} and \eqref{eq:line}, it suffices to prove that any operator $\mathcal{L}_{\phi_{\nabla \varphi}^1}$, $\nabla\varphi\in \mathfrak{G}$, is $L^2$-STAR. 

%%%%%%%%%%%%%%%%%%%%%%%%%%%%%%%%%%%%%%%%%%%%%%%%%%%
\appendix
\section{Appendix}

%--------------------------------------------------
\subsection{Well posedness for piecewise constant controls}
\label{App:WP}

In light of Propositions \ref{prop:self-adjointness} and
\ref{prop:kato-rellich} below, we can define the solutions 
of the two systems 
%%
%\eqref{eq:oscillator_BIS},
\eqref{eq:torus}, 
\eqref{eq:line},
associated with piecewise constant controls, by composition of time-independent unitary propagators associated with self-adjoint operators (see, e.g., \cite[Definition p.256 \& Theorem VIII.7]{rs1}). For instance, for system \eqref{eq:line}, 
given a subdivision $0=t_0<\dots<t_N=T$,
a piecewise constant control 
$u:[0,T]\to \mathbb{R}^{d+1}$ defined as 
$u(t)=(u^{j}_1,\dots,u^{j}_{d+1})\in\mathbb{R}^{d+1}$ when 
$t \in [t_{j-1},t_j]$, 
and an initial condition $\psi_0 \in L^2(\R^d,\C)$,
the solution $\psi\in C^0([0,T],L^2(\R^d,\mathbb{C}))$ of \eqref{eq:line} is defined by
\begin{equation} \label{eq:propagator}
\psi(t;u,\psi_0)
=e^{-i(t-t_{j-1})A_j}
e^{-i \tau_{j-1}A_{j-1}}\\
\dots
e^{-i\tau_1 A_1} \psi_0
\end{equation}
where $\tau_l=(t_l-t_{l-1})$ and
$A_l = -\Delta+V+\sum_{k=1}^du_k^{l}x_k-u_{d+1}^{l} e^{-|x|^2/2}$
for $l=1,\dots,N$.

\begin{proposition}\label{prop:self-adjointness}{\cite[Corollary page 199]{rs2}}
%\begin{itemize}
%\item[(i)] 
Let 
%$A\in C^1(\R^d,\R^d)$ and 
$V$ satisfying (\ref{Hyp:V_transp}).
%$V\in L^2_{\rm loc}(\R^d,\R)$. If there exist $a,b\geq 0$ such that
%$$V(x)\geq -a\|x\|^2-b,$$
Then 
%$-\Delta_A+V$ 
$-\Delta+V$
is essentially self-adjoint on $C^\infty_c(\R^d,\C)$.
%\item[(ii)] (Reed-Simon 2: Theorem X.34) Let $V=V_1+V_2$, $V_1\geq 0, V_1\in L^2_{\rm loc}(\R^m,\R)$,$V_2$ $\Delta$-bounded with relative bound $a<1$. Let $a_j\in C^1(\R^d,\R)$. Then, 
%$$-\sum_{j=1}^d(\partial_{x_j}-ia_j)^2+V$$
%is essentially self-adjoint on $C^\infty_c(\R^d,\C)$.
%\end{itemize}
\end{proposition}

\begin{definition}
Given two densely defined linear operators $A$ and $B$ with domains $D(A)$ and $D(B)$ on an Hilbert space $\mathcal{H}$, $B$ is said to be $A$-bounded if $D(A)\subset D(B)$ and
there exist $a,b\geq 0$ such that for all $\psi\in D(A)$
$$\|B\psi\|<a\|A\psi\|+b\|\psi\|.$$
The infimum of such $a$ is called the relative bound of $B$.
\end{definition}

\begin{proposition}\label{prop:kato-rellich}{(Kato-Rellich Theorem)\cite[Theorem X.12]{rs2}}
If $A$ is self-adjoint and $B$ is symmetric and $A$-bounded with relative bound $a<1$, then $A+B$ is self-adjoint on $D(A)$ and essentially self-adjoint on any core of $A$.
\end{proposition}

The assumption $V\in L^\infty$ (resp. $V$ satisfies \eqref{Hyp:V_transp}) given in Theorem \ref{Main_Thm_torus} (resp. Theorem \ref{thm:global-euclidean}) guarantees that $-\Delta+V$  is essentially self-adjoint on $C^\infty(\T^d)$ (resp. $C^\infty_c(\R^d)$). More in general, as proved by Gaffney in \cite{gaffney}, if $V\in L^\infty(M,\R)$ and $M$ is complete, then $-\Delta+V$ is essentially self-adjoint on $C^\infty_c(M,\C)$.

%-------------------------------------------------
\subsection{Semigroup structure of $L^2$-STAR operators} \label{subsec:STAR_op}

In this section, we prove Lemma \ref{lem:reachable-operators} for sake of completeness.

\begin{proof}[Proof of Lemma \ref{lem:reachable-operators}]
\emph{Step 1: Semi-group structure.}
Let $L_1, L_2$ be  $L^2$-STAR operators.
Let $\psi_0 \in \mathcal{S}$ and $\varepsilon>0$. For $j=1,2$, there exists $u_j \in PWC(0,T_j)$ such that
\begin{equation} 
\| \psi(T_2; u_2, L_1 \psi_0 ) -  L_2 L_1 \psi_0 \|_{L^2}< \varepsilon,
\qquad
\| \psi(T_1;u_1,\psi_0)- L_1 \psi_0 \|_{L^2} < \varepsilon .
\end{equation}
We consider the concatenation
$u:=u_1 \sharp u_2 \in PWC(0,T_1+T_2)$. Then
$$\begin{aligned}
& \| \psi(T_1+T_2;u,\psi_0) -  L_2 L_1 \psi_0 \|_{L^2} 
\\ \leq & \| \psi(T_2;u_2,\psi(T_1;u_1,\psi_0))- \psi(T_2;u_2, L_1 \psi_0) \|_{L^2} + \| \psi(T_2;u_2, L_1 \psi_0)-   L_2 L_1 \psi_0\|_{L^2}
\\  \leq &  \| \psi(T_1;u_1,\psi_0)-  L_1 \psi_0 \|_{L^2} +  
\varepsilon
<  2 \varepsilon.
\end{aligned}$$

\noindent \emph{Step 2: Stability by strong convergence.} Let $(L_n)_{n\in\N}$ be a sequence of $L^2$-STAR operators and $L$ be an operator on $L^2(M,\C)$ such that $(L_n)_{n\in\N}$ strongly converges towards $L$, i.e. for every $\psi \in L^2(M,\C)$, $\|(L_n-L)\psi\|_{L^2} \to 0$ as $n \to \infty$.

 Let $\psi_0 \in \mathcal{S}$ and $\varepsilon>0$. There exists $n \in \N$ such that 
$\|(L_n -L)\psi_0\|_{L^2} < \varepsilon/2$. There exists $T \in [0,\epsilon]$, and $u\in \text{PWC}(0,T)$ such that 
$\| \psi(T;u,\psi_0)- L_n \psi_0 \|_{L^2} < \varepsilon/2$. Then 
$
\| \psi(T;u,\psi_0)-L \psi_0 \|_{L^2}
\leq \| \psi(T;u,\psi_0)- L_n \psi_0 \|_{L^2} + 
\| (L_n -L)\psi_0\|_{L^2} < \varepsilon.
$

\end{proof}

%-------------------------------------------------
\subsection{Transitive group action of ${\rm Diff}_c^0(M)$ on densities} \label{subsec:Moser}

In this section, we give precise statements about the transitivity of the group action of ${\rm Diff}_c^0(M)$ (see \eqref{def:Diffc0_action}) on an appropriate set of positive densities over $M$, proved by Moser \cite{moser}. Proposition \ref{lem:transitivity-torus} concerns $M=\T^d$ and is valid more generally on any compact connected manifold $M$.
Proposition \ref{Prop:Moser_Rd} concerns $M=\R^d$.
In any case, the diffeomorphism $P$ is given by simple integrations.

\begin{proposition} \label{lem:transitivity-torus}
Let $\rho_0, \rho_1\in C^\infty(\T^d,(0,\infty))$ be such that $\|\rho_0\|_{L^2}=\|\rho_1\|_{L^2}=1$. 
There exists $P \in {\rm Diff}^0_c(\T^d)$ such that $\rho_1=\mathcal{L}_P(\rho_0)$.
\end{proposition}

\begin{proof} \emph{Step 1: We prove the result when $\rho_0$ is constant on $\T^d$ i.e. there exists 
$P \in {\rm Diff}^0_c(\T^d)$ such that, for every 
$x \in \T^d$,
$|J_P(x)|^{1/2} (2\pi)^{-d/2}  = \rho_1(x)$.} 
For $j \in \{1,\dots,d\}$, we define smooth positive functions $I_j, h_j: \T^{j} \to (0,\infty)$ by
$$I_d=\rho_1^2 \qquad \text{ and } \qquad 
I_j(x_1,\dots,x_j)=\int_{\T^{d-j}} \rho_1^2(x_1,\dots,x_j,y) dy \quad \text{ for } j\in \{1,\dots,d-1\},$$
$$h_1(x_1)=I_1(x_1) \qquad \text{ and } \qquad
h_j(x_1,\dots,x_{j})=\frac{I_{j}(x_1,\dots,x_j)}{I_{j-1}(x_1,\dots,x_{j-1})} \quad \text{ for } j\in\{2,\dots d\}.$$
Then, for every $j \in \{ 1,\dots d \}$ and $(x_1,\dots,x_j) \in \T^d$,
\begin{equation} \label{prop:hj}
    h_1(x_1)\dots h_j(x_1,\dots,x_j)= I_{j}(x_1,\dots,x_j)
\qquad \text{ and } \qquad
\int_{\T} h_j(x_1,\dots,x_{j-1},\sigma) d\sigma = 1.
\end{equation}
Identifying $\T^d$ with $[0,2\pi[^d$, we define the $C^{\infty}$-map $P:\T^d \rightarrow \T^d$ by
$$P(x_1,\dots,x_d)=\left(
2\pi \int_0^{x_1} h_1 (\sigma)d\sigma,
2\pi \int_0^{x_2} h_2(x_1,\sigma) d\sigma, \dots,
2\pi \int_0^{x_d} h_d(x_1,\dots,x_{d-1},\sigma)d\sigma \right)$$
The 2nd equality of (\ref{prop:hj}) proves that $P$ maps $\T^d$ into $\T^d$ and is surjective on $\T^d$. $P$ is also injective on $\T^d$ because the functions $h_j$ are continuous and positive. 

The same arguments prove that, for every $\alpha \in [0,1]$, the map $\alpha {\rm Id}_{\T^d}+(1-\alpha)P$ is a $C^{\infty}$-diffeomorphism of $\T^d$
(replace $h_j$ by $\frac{\alpha}{2\pi} + (1-\alpha)h_j$ which still satisfies the 2nd equality of (\ref{prop:hj})).
The continuous path $\alpha \in [0,1] \mapsto 
\alpha {\rm Id}_{\T^d}+(1-\alpha)P$  connects $P$ with ${\rm Id}_{\T^d}$ thus $P\in {\rm Diff}^0_c(\T^d)$.

Finally, for every $x=(x_1,\dots,x_d) \in \T^d$, thanks to the triangular form of
$DP(x)$, we obtain
$J_P(x)=(2\pi)^d h_1(x_1) \dots h_d(x_1,\dots,x_d)=(2\pi)^d \rho_1^2(x)$.

\medskip

\noindent \emph{Step 2: Conclusion.} By Step 1, there exist
$P_0, P_1 \in {\rm Diff}^0_c(\T^d)$ such that, for every 
$x \in \T^d$,
$|J_{P_0}(x)|^{1/2}   = (2\pi)^{d/2} \rho_0(x)$ and
$|J_{P_1}(x)|^{1/2}   = (2\pi)^{d/2} \rho_1(x)$. Then $P:=P_0^{-1} \circ P_1 \in  {\rm Diff}^0_c(\T^d)$ and, by the chain rule, for every $x \in \T^d$
$$|J_P(x)|^{1/2} \rho_0(P(x))=|J_{P_0} (P(x))^{-1} J_{P_1}(x)|^{1/2} \rho_0(P(x))= \rho_1(x).$$
\end{proof}

\begin{definition}[$Q_R^d$]
    For $R\in (0,\infty)$ and $d \in \N^*$, we use the notation 
    $Q_R^d:=(-R,R)^d$.
\end{definition}

\begin{proposition} \label{Prop:Moser_Rd}
    Let $\rho_0, \rho_1 \in C^{\infty}(\R^d,[0,\infty))$ and $0<R'<R$ be such that 
\begin{equation} \label{hyp_Moser_line}
\rho_0, \rho_1>0 \text{ on } Q_R^d, \qquad
\int_{Q_R^d} \rho_0^2 = \int_{Q_R^d} \rho_1^2=1, \qquad
\rho_0=\rho_1 \text{ on } \R^d \setminus Q_{R'}^d.
\end{equation}
    Then, there exists $P \in {\rm Diff}^0_c(\R^d)$ such that $\rho_1=\mathcal{L}_P(\rho_0)$.    
\end{proposition}

\begin{proof}
For $\ell \in \{0,1\}$, we define smooth positive functions $I_j^{\ell}, h_j^{\ell}: Q_R^{j} \to (0,\infty)$ by
$$I_d^{\ell}=\rho_{\ell}^2 \qquad \text{ and } \qquad 
I_j^{\ell}(x_1,\dots,x_j)=\int_{Q_R^{d-j}} \rho_{\ell}^2(x_1,\dots,x_j,y) dy \quad \text{ for } j\in \{1,\dots,d-1\},$$
$$h_1^{\ell}(x_1)=I_1^{\ell}(x_1) \qquad \text{ and } \qquad
h_j^{\ell}(x_1,\dots,x_{j})=\frac{I_{j}^{\ell}(x_1,\dots,x_j)}{I_{j-1}^{\ell}(x_1,\dots,x_{j-1})} \quad \text{ for } j\in\{2,\dots d\}.$$
Then, for every  $\ell \in \{0,1\}$, $j \in \{ 1,\dots d \}$ and 
$(x_1,\dots,x_j) \in Q_{R}^{j}$
\begin{equation} \label{prop:hj}
    h_1^{\ell}(x_1)\dots h_j^{\ell}(x_1,\dots,x_j)= I_{j}^{\ell}(x_1,\dots,x_j)
\qquad \text{ and } \qquad
\int_{-R}^{R} h_j^{\ell}(x_1,\dots,x_{j-1},\sigma) d\sigma = 1.
\end{equation}
The map $P_{\ell}:Q_R^d \rightarrow Q_R^d$ defined by
$$P_{\ell}(x_1,\dots,x_d)=\left(
-R+2R \int_{-R}^{x_1} h_1 (\sigma)d\sigma,
%-R+2R \int_{-R}^{x_2} h_2(x_1,\sigma) d\sigma, 
\dots,
-R+2R \int_{-R}^{x_d} h_d(x_1,\dots,x_{d-1},\sigma)d\sigma \right)$$
is a  $C^{\infty}$-diffeormphism of $Q_R^d$ isotopic to identity, for the same reasons as in the previous proof. 
The assumption $\rho_0=\rho_1$ on $Q_R^d \setminus Q_{R'}^d$
implies $I_j^0=I_j^1$ and $h_j^{0}=h_j^{1}$ on $Q_R^j \setminus Q_{R'}^j$ thus
$P_0=P_1$ on $Q_R^d \setminus Q_{R'}^d$.
We conclude with $P(x)=P_0^{-1} \circ P_1(x)$ for $x \in Q_R^d$ and $P(x)=x$ for $x \in \R^d \setminus Q_R^d$.
\end{proof}

%--------------------------------------------------
\subsection{Some density results}

\begin{lemma} \label{Lem:dense_tore}
    The set 
    $\frak{D}_1(\T^d):=\{ \rho e^{i \phi} ; 
    \phi \in L^{2}(\T^d,\R),
    \rho \in C^{\infty}(\T^d,(0,\infty)), 
    \|\rho\|_{L^2}=1 
    \}$
    is dense in $\mathcal{S}$ for the $\|.\|_{L^2}$-topology.
    Thus $\frak{D}(\T^d):=\frak{D}_1(\T^d) \times \frak{D}_1(\T^d)$ is dense in 
    $\mathcal{S} \times \mathcal{S}$ for the $\|.\|_{L^2}$-topology.
\end{lemma}

\begin{proof}
Let $\psi \in \mathcal{S}$.
Then $\rho:=|\psi| \in L^2(\T^d, [0,\infty))$ and $\|\rho\|_{L^2}=1$.
For $\epsilon>0$, we have $\max\{\epsilon,\rho\} \in L^2(\T^d,(0,\infty))$ and $\|\rho-\max\{\epsilon,\rho\}\|_{L^2} \leq \epsilon (2\pi)^d$.
By considering a convolution product 
of $\max\{\epsilon,\rho\}$
with a non negative kernel, 
and an appropriate renormalization,
we obtain 
$\rho^{\epsilon} \in C^{\infty}(\T^d,(0,\infty))$ such that 
$\|\rho^{\epsilon}\|_{L^2}=1$ and
$\|\rho-\rho^{\epsilon}\|_{L^2} \to 0$ as $\epsilon \to 0$.

The function $\text{arg}(\psi)$ is well defined on the set $\{ \rho>0\}$. Let $\phi$ be its extension by $0$ on $\T^d$, then $\phi$ is a measurable function on $\T^d$ with values in $[-\pi,\pi)$ thus 
$\phi \in L^2(\T^d)$. Finally
$\| \psi - \rho^{\epsilon} e^{i \phi} \|_{L^2} =
\| \rho - \rho^{\epsilon}  \|_{L^2} \to 0$ as $\epsilon \to 0$.    
\end{proof}

\begin{lemma} \label{Lem:dense_line}
The set $\frak{D}(\R^d)$, defined below, is dense in $\mathcal{S} \times \mathcal{S}$ for the $\|.\|_{L^2}$-topology,
$$\frak{D}(\R^d):=\{ (\rho_0 e^{i\phi_0},\rho_1 e^{i \phi_1});
\phi_j \in L^2(\R^d,\R),
\rho_j \in C^{\infty}(\R^d,[0,\infty)),
\exists 0<R'<R \text{ s.t. } 
(\ref{hyp_Moser_line}) \text{ holds}
\}.$$
\end{lemma}

\begin{proof}
Let $(\psi_0,\psi_1)\in \mathcal{S} \times \mathcal{S}$. 
One may assume these functions are compactly supported: there exist $R'>0$ such that $\text{Supp}(\psi_j) \subset Q_{R'}^d$.
Then $f_j:=|\psi_j| \in L^2(\R^d,[0,\infty))$ is supported in $Q_{R'}^d$. By considering a convolution product of $\max\{ f_j,\epsilon\} 1_{Q_{R'}^d}$ with a non negative kernel, 
and an appropriate renormalization,
we obtain $f_j^{\epsilon} \in C^{\infty}_c(\R^d,[0,\infty))$
such that
$\text{Supp}(f_j^{\epsilon}) = \overline{Q_{R'+\epsilon}^{d}}$, 
$\|f_j^{\epsilon}\|_{L^2}=1$,
$f_j^{\epsilon}>0$ on $Q_{R'+\epsilon}$ and
$\|f_j-f_j^{\epsilon}\|_{L^2} \to 0$ when $\epsilon \to 0$.
Let $R:=R'+1$ and
$\chi \in C^{\infty}_c(\R^d,[0,\infty)$ such that
$\text{Supp}(\chi)=\overline{Q_R^d} \setminus Q_{R'}^d$ and
$\chi>0$ on $Q_{R}^d \setminus \overline{Q_{R'}^d}$.
For $\epsilon, \delta>0$ small enough, there exists $\alpha_{\epsilon,\delta}>0$ such that
\begin{equation} \label{eq:a}
\frac{1}{\| f_0^{\epsilon}+\delta \chi\|_{L^2}} = 
\frac{\alpha_{\epsilon,\delta}}{\| f_1^{\epsilon}+\delta \alpha_{\epsilon,\delta} \chi\|_{L^2}}
\end{equation}
because this equation can be re-formulated as 
$a_{\epsilon,\delta} \alpha^2 + 2b_{\epsilon,\delta} \alpha + c_{\epsilon,\delta} =0$ where
$$a_{\epsilon,\delta}:=\|f_0^{\epsilon}+\delta \chi\|_{L^2}^2-\delta^2 \|\chi\|_{L^2}^2,
\qquad 
b_{\epsilon,\delta}:= -2\delta \int_{Q_{R}^d} f_1^{\epsilon} \chi,
\qquad
c_{\epsilon,\delta}:= - 1$$
thus
$\Delta_{\epsilon,\delta}:=b_{\epsilon,\delta}^2-4a_{\epsilon,\delta}c_{\epsilon,\delta} \to 4$ 
as $(\epsilon,\delta) \to 0$
and
$\alpha_{\epsilon,\delta}:=(-b_{\epsilon,\delta}+\sqrt{\Delta_{\epsilon,\delta}})/2 a_{\epsilon,\delta}$ gives the conclusion.
For $\epsilon, \delta>0$ small enough, we define
$$\rho_0^{\epsilon,\delta}(x):=\frac{f_0^{\epsilon}(x)+\delta \chi(x)}{\| f_0^{\epsilon}+\delta \chi\|_{L^2} }
\quad
\text{ and }
\quad
\rho_1^{\epsilon}(x):=\frac{f_1^{\epsilon}(x)+ \delta \alpha_{\epsilon,\delta}  \chi(x)}{\| f_1^{\epsilon}+\delta  \alpha_{\epsilon,\delta} \chi\|_{L^2} }.$$
Then $\rho_j \in C^{\infty}_c(\R^d,[0,\infty))$,
$\rho_j>0$ on $Q_R^d$, 
$\|\rho_j\|_{L^2(Q_R^d)}=1$
and (\ref{eq:a}) ensures that
$\rho_0=\rho_1$ on $Q_R^d \setminus Q_{R'+\epsilon}$.
Moreover, $\|f_j-\rho_j^{\epsilon,\delta}\|_{L^2} \to 0$ as $(\epsilon,\delta)\to 0$.

The function $\text{arg}(\psi_j)$ is well defined on the set $\{ |\psi_j|>0\} \subset Q_R^d$. Let $\phi_j$ be its extension by $0$ on $\R^d$, then $\phi_j$ is a measurable function on $\R^d$ with values in $[-\pi,\pi)$ and compactly supported thus 
$\phi_j \in L^2(\R^d)$. Finally
$\| \psi_j - \rho_j^{\epsilon} e^{i \phi_j} \|_{L^2}  \to 0$ as $\epsilon \to 0$. 
\end{proof}

\medskip

\textbf{Acknowledgments.}
The authors would like to thank Mario Sigalotti for stimulating discussions.

 Karine Beauchard acknowledges support from grants ANR-20-CE40-0009 (Project TRECOS) and ANR-11-LABX-0020 (Labex Lebesgue), as well as from the Fondation Simone et Cino Del Duca -- Institut de France.

Eugenio Pozzoli thanks the SMAI for supporting and the CIRM for hosting the BOUM project "Small-time controllability of Liouville transport equations along an Hamiltonian field", where some ideas of this work were conceived.

This research has been funded in whole or in part by the French National Research Agency (ANR) as part of the QuBiCCS project "ANR-24-CE40-3008-01".

This project has received financial support from the CNRS through the MITI interdisciplinary programs.

\bibliographystyle{siamplain}
\bibliography{references}

%%%%%%%%%%%%%%%%%%%%%%%%%%%%%
\end{document}